\documentclass[11pt]{amsart}
\usepackage{amssymb, times}
\usepackage{graphicx}
\usepackage{amsfonts}
\usepackage{amsmath}
\usepackage{amsthm}
\usepackage{fancyhdr}
\usepackage{indentfirst}
\usepackage{hyperref}
\usepackage{comment}
\usepackage{color}
\usepackage{subcaption}
\usepackage{epsfig}
\usepackage{pst-grad} 
\usepackage{pst-plot} 
\usepackage[space]{grffile} 
\usepackage{etoolbox} 
\usepackage{mathrsfs} 
\usepackage{enumitem} 

\newcommand{\Ac}{\mathcal{A}^c}
\newcommand{\al}{\alpha}

\newcommand{\V}{\mathcal{V}}
\newcommand{\R}{\mathbb{R}}
\newcommand{\N}{\mathbb{N}}
\newcommand{\mH}{\mathcal{H}}
\newcommand{\F}{\mathcal{F}}

\newcommand{\sB}{\widetilde{B}}
\newcommand{\C}{\mathcal{C}}

\newcommand{\lc}{\scalebox{1.8}{$\llcorner$}}
\newcommand{\La}{\Lambda}
\newcommand{\si}{\sigma}
\newcommand{\Si}{\Sigma}
\newcommand{\de}{\delta}

\newcommand{\ep}{\epsilon}
\newcommand{\pr}{\prime}
\newcommand{\Om}{\Omega}

\newcommand{\Ga}{\Gamma}
\newcommand{\ga}{\gamma}

\newcommand{\lap}{\triangle}
\newcommand{\ti}{\tilde}

\newcommand{\Z}{\mathcal{Z}}
\newcommand{\mS}{\mathcal{S}}
\newcommand{\f}{\mathbf{f}}

\newcommand{\M}{\mathbf{M}}
\newcommand{\bL}{\mathbf{L}}
\newcommand{\bI}{\mathbf{I}}
\newcommand{\mf}{\mathbf{f}}

\newcommand{\mF}{\mathbf{F}}
\newcommand{\mR}{\mathcal{R}}

\newcommand{\X}{\mathfrak{X}}
\newcommand{\sA}{\mathscr{A}}

\newcommand{\btau}{\boldsymbol\tau}
\newcommand{\bleta}{\boldsymbol{\eta}}
\newcommand{\n}{\mathbf{n}}

\newcommand{\md}{\mathbf{d}}

\newcommand{\rom}[1]{\expandafter\romannumeral #1}
\newcommand{\Rom}[1]{\uppercase\expandafter{\romannumeral #1}}

\newcommand{\VarTan}{\operatorname{VarTan}}
\newcommand{\spt}{\operatorname{spt}}
\newcommand{\dist}{\operatorname{dist}}

\newcommand{\vol}{\operatorname{Vol}} 

\newcommand{\Area}{\operatorname{Area}}

\newcommand{\Clos}{\operatorname{Clos}}
\newcommand{\interior}{\operatorname{int}}

\begin{document}

\newtheorem{theorem}{Theorem}[section]
\newtheorem{proposition}[theorem]{Proposition}
\newtheorem{corollary}[theorem]{Corollary}

\newtheorem{claim}{Claim}

\theoremstyle{remark}
\newtheorem{remark}[theorem]{Remark}

\theoremstyle{definition}
\newtheorem{definition}[theorem]{Definition}

\theoremstyle{plain}
\newtheorem{lemma}[theorem]{Lemma}

\numberwithin{equation}{section}

\title[Min-max theory for CMC hypersurfaces]{Min-max theory for constant mean curvature hypersurfaces}

\author[Xin Zhou]{Xin Zhou}
\address{Department of Mathematics, University of California Santa Barbara, Santa Barbara, CA 93106, USA}
\email{zhou@math.ucsb.edu}

\author[Jonathan J. Zhu]{Jonathan J. Zhu}
\address{Department of Mathematics, Harvard University, Cambridge, MA 02138, USA}
\email{jjzhu@math.harvard.edu}

\maketitle

\pdfbookmark[0]{}{beg}

\begin{abstract}
In this paper, we develop a min-max theory for the construction of constant mean curvature (CMC) hypersurfaces of prescribed mean curvature in an arbitrary closed manifold. As a corollary, we prove the existence of a nontrivial, smooth, closed, almost embedded, CMC hypersurface of any given mean curvature $c$. Moreover, if $c$ is nonzero then our min-max solution always has multiplicity one.%
\end{abstract}

\setcounter{section}{-1}

\section{Introduction}
\label{S:intro}

A hypersurface $\Sigma^n$ in a Riemannian manifold $M^{n+1}$ has constant mean curvature if it is a critical point of the area functional amongst variations preserving the enclosed volume. Equivalently, hypersurfaces of constant mean curvature $c$ are critical points of the Lagrange-multiplier functional 
\begin{equation}
\label{E:Ac0}
\mathcal{A}^c = \Area - c\vol. 
\end{equation}
Constant mean curvature (CMC) hypersurfaces constitute a classical and extensively-studied topic in differential geometry, and play an essential role in many areas, from isoperimetric problems \cite{ros05} to the modeling of interface phenomena \cite{Lobaton-Salamon07, Lopez05} and to general relativity \cite{HY96, QT07, CPP10}. Despite the classical nature of the problem, relatively few examples of closed CMC hypersurfaces were known, even for $n=2$, until the breakthrough work of Wente \cite{Wente86}. Many attempts have been made to construct more CMC hypersurfaces, especially with \textit{prescribed} constant mean curvature, c.f. \cite{Heinz54, Hildebrandt70, struwe85, Kap90, Ye91, Duzaar-Steffen96, pacard05, MMP06, Rosenberg-Smith16}. However, these works left wide open the question of which values may be prescribed - that is, for which constants $c$ does there exist a closed hypersurface of constant mean curvature $c$?

In this article we construct, via a min-max approach, nontrivial closed CMC hypersurfaces of \textit{any} prescribed mean curvature, in any smooth closed Riemannian manifold $M^{n+1}$ of dimension at most seven. (The dimensional restriction arises from, and matches with, the well-known regularity theory for minimal hypersurfaces \cite{SSY75, SS81}; see also \cite{wick14}.)

\begin{theorem}
\label{thm:mainA}
Let $M^{n+1}$ be a smooth, closed Riemannian manifold of dimension $3\leq n+1\leq 7$. Given any $c\in\mathbb{R}$, there exists a nontrivial, smooth, closed, almost embedded hypersurface $\Sigma^n$ of constant mean curvature $c$. 
\end{theorem} 

Here we say that an immersed hypersurface $\Sigma$ is almost embedded if $\Sigma$ locally decomposes into smoothly embedded components that (pairwise) lie to one side of each other. That is, the sheets may touch but not cross; the example of touching spheres shows that this regularity is optimal. In particular, almost embedded hypersurfaces are automatically Alexandrov embedded.

We want to compare our result with a classical problem by Arnold \cite[page 395]{Arnold04} and Novikov \cite[Section 5]{Novikov82} on the periodic orbits of a charged particle in a magnetic field on a topological two sphere. It is conjectured that there exist closed embedded curves of any prescribed constant geodesic curvature. This conjecture remains open, and we refer to \cite{Ginzburg96, Rosenberg-Schneider11, Schneider11} for more backgrounds and some partial results of this conjecture. Our result can be viewed as a complete resolution of the higher dimensional analog of Arnold-Novikov conjecture.

\begin{remark}
For $c\neq 0$, we can prove that our min-max procedure converges to the constructed hypersurface $\Sigma$ with multiplicity 1. This is a stark and surprising contrast to the minimal ($c=0$) case, for which the min-max multiplicity 1 conjecture is a fundamental open problem \cite{MN16}. 
\end{remark}

The existence problem for CMC hypersurfaces has been studied from a number of perspectives. The boundary value problems were substantially developed by Heinz \cite{Heinz54}, Hildebrandt \cite{Hildebrandt70}, Struwe \cite{struwe85, struwe88}, etc. using the mapping method, and by Duzaar-Steffen \cite{Duzaar-Steffen96} using geometric measure theory, while both methods can only produce CMC hypersurfaces whose mean curvatures satisfy certain upper bound. For the case of closed CMC hypersurfaces, the more classical approach is to minimize the area functional amongst volume-preserving variations, that is, to solve the isoperimetric problem for a given volume. Indeed, for each fixed volume there exists a smooth minimizer (up to a singular set of codimension 7; see for instance \cite{Almgren76, Mo03, ros05}). However, this approach does not yield any control on the value of the mean curvature. 

Another class of approaches relies on perturbative methods. Given a closed minimal hypersurface, one may deform it to a CMC hypersurface, but only for very small values of the mean curvature. On the other end, one attempts to construct foliations by closed CMC hypersurfaces near minimal submanifolds of strictly lower dimension. This program was carefully implemented in various cases by Ye \cite{Ye91}, Mahmoudi-Mazzeo-Pacard \cite{MMP06}, and others (see the survey article \cite{pacard05}). The hypersurfaces produced by this approach necessarily have large mean curvatures, which in fact diverge as the hypersurfaces condense onto the minimal submanifold. 

We also mention the delicate gluing procedures pioneered by Kapouleas \cite{Kap90} as well as the degree theory developed by Rosenberg-Smith \cite{Rosenberg-Smith16}. These provide important examples of CMC hypersurfaces, but the former is typically restricted by the availability of known solutions, whilst the latter can only produce CMC hypersurfaces of fairly large mean curvature greater than some threshold depending on ambient manifolds. Finally, we remark that Meeks-Mira-Perez-Ros \cite{MMPR13, MMPR17} were able to determine, in the special case of homogeneous ambient 3-manifolds, precisely the values for which there exists a CMC 2-sphere with the specified mean curvature.

\vspace{0.3cm}
In order to prove Theorem \ref{thm:mainA}, we instead study CMC hypersurfaces from the perspective of the $\mathcal{A}^c$-functional. It is easy to see that the simplest method, minimization, does not succeed in detecting a nontrivial critical point for the $\Ac$-functional. In fact, the minimizer of $\Ac$ among domains $\Om$ in $M$ with smooth boundary is always the total manifold $M$, as $\Ac(M)=-c\vol(M)\leq \Ac(\Om)$. Therefore, the min-max method becomes the natural way to find nontrivial critical points of $\Ac$. (Note that minimization method does succeed for the Plateau problem with fixed boundary \cite{Duzaar-Steffen96}.)

For finding critical points of the area functional - that is, minimal hypersurfaces - the min-max method has been greatly successful. In \cite{AF62}, Almgren initiated a celebrated program to develop a variational theory for minimal submanifolds in Riemannian manifolds of any dimension and co-dimension using geometric measure theory, namely the min-max theory for minimal submanifolds. He was able to prove the existence of a nontrivial weak solution as stationary integral varifolds \cite{AF65}. Higher regularity was established in the co-dimension-one case by the seminal work of Pitts \cite{P81} (for $2\leq n\leq 5$) and later extended by Schoen-Simon [37] (for $n\geq 6$).  Colding-De Lellis \cite{CD03} established the corresponding theory using smooth sweepouts based on ideas of Simon-Smith \cite{Sm82}. Indeed, the preceeding body of work completely resolved the $c=0$ case of Theorem \ref{thm:mainA}.

Very recently, in a series of works, Marques-Neves \cite{MN14, Agol-Marques-Neves16, MN17} found surprising applications of the Almgren-Pitts min-max theory to solve a number of longstanding open problems in geometry, including their celebrated proof of the Willmore conjecture. Due to these tremendous successes, there have been a vast number of developments of this program in various contexts, including \cite{DT13, Montezuma16, Guaraco15, Ketover-Zhou15, MN16, Liokumovich-Marques-Neves16, Chambers-Liokumovich16, DR16, LiZ16, Ketover16, Song17}. In this regard, our work represents a natural extension of the min-max method to the CMC setting. 

\subsection{Min-max procedure} 

We now give a heuristic overview of our min-max method. In the main proofs, for technical reasons we will work with discrete families as in Almgren-Pitts, but here we will describe the key ideas using continuous families to elucidate those ideas.

Let $M$, $c$ be as in Theorem \ref{thm:mainA}. We only need to consider $c>0$. The $\mathcal{A}^c$ functional (\ref{E:Ac0}) is defined on open sets $\Omega$ with smooth boundary by  $\mathcal{A}^c(\Omega) = \Area(\partial \Omega) - c \vol(\Omega)$. Denote $I=[0, 1]$. Consider a continuous $1$-parameter family of sets with smooth boundary
\[ \{\Om_x: x\in I\}, \text{ with $\Om_0=\emptyset$ and $\Om_1=M$}. \]

Fix such a family $\{\Om^0_x\}$, and consider its homotopy class $[\{\Om^0_x\}]=\big{\{ } \{\Om_x\}\sim \{\Om^0_x\} \big{\}}$.
The {\em $c$-min-max value} is defined as
\[ \bL^c=\inf_{\{\Om_x\}\sim \{\Om^0_x\}}\max\{\Ac(\Om_x): x\in I\}. \]

A sequence $\{\{\Om^i_x\}: i\in\N\}$ with $\max_{x\in I} \Ac(\Om^i_x)\to \bL^c$ is typically called a {\em minimizing sequence}, and any sequence $\{\Om^i_{x_i}: x_i\in (0, 1), i\in N\}$ with $\Ac(\Om^i_{x_i})\to \bL^c$ is called a {\em min-max sequence}. 

Our main result (for the precise statement see Theorem \ref{thm:main-body}) then says that there is a nice minimizing sequence $\{\{\Om^i_x\}: i\in\N\}$, and some min-max sequence $\{\Om^i_{x_i}: x_i\in (0, 1), i\in N\}$, such that: 
\begin{theorem}
\label{T: main}
The sequence $\partial\Om^i_{x_i}$ converges as varifolds with multiplicity one to a nontrivial, smooth, closed, almost embedded hypersurface $\Si$ of constant mean curvature $c$.
\end{theorem}

Our proof broadly follows the Almgren-Pitts scheme, but with several important difficulties. This scheme proceeds generally as follows:

\begin{itemize}
\item Construct a sweepout with positive width, and extract a minimizing sequence;
\item Apply a `tightening' map to construct a new sequence whose varifold limit satisfies a variational property and an `almost-minimizing' property;
\item Use these properties to construct `replacements' on annuli which must be regular;
\item Apply successive concentric annular replacements to the min-max limit and show that they coincide with each other, and hence extend to the center;
\item Show that the min-max limit coincides with the replacement near the center.
\end{itemize}

We in fact show that $\bL^c$ is positive on any sweepout, as a consequence of the isoperimetric inequality for small volumes (see Theorems \ref{T:Isoperimetric areas} and \ref{T:existence of nontrivial sweepouts}). 

For the tightening, in the minimal ($c=0$) case one shows that for the tightened sequence, any min-max (varifold) limit must be stationary, that is, a weak solution in the sense of first variations. In the CMC setting it is an important, yet subtle, issue to determine the correct variational property to replace stationarity. For instance, the $\mathcal{A}^c$ functional is not well-defined on varifolds, so it is difficult to formulate a notion of weak solution for its critical points. To overcome this issue, we utilize the property of {\em $c$-bounded first variation}. This notion is a generalisation of bounded mean curvature, and is loose enough to be satisfied by our min-max limit $V$ (after tightening) whilst providing enough control to develop the regularity theory. In particular, varifolds with $c$-bounded first variation satisfy a uniform monotonicity formula, and any blowups are stationary. Furthermore, we formulate the property of being {\em $c$-almost minimizing}, which is inspired by the almost-minimizing property of Almgren-Pitts, with the area functional replaced by the $\mathcal{A}^c$ functional (see Definition \ref{D:c-am-varifolds}). 

To construct our so-called {\em $c$-replacements}, we solve a series of constrained minimization problems for the $\mathcal{A}^c$-functional in a subset $U\subset M$. Each $\mathcal{A}^c$-minimizer is an open set $\Omega_i^*$ with stable, regular CMC boundary in $U$, and the $c$-replacement $V^*$ is obtained as the varifold limit $\lim |\partial \Omega_i^*|$. At this point, $V^*$ gains regularity by classical curvature estimates, but in contrast to the minimal case, $V^*$ is merely almost embedded since we only have a \textit{one}-sided maximum principle for CMC hypersurfaces (see Lemma \ref{L:classical MP}). Nevertheless, the one-sided maximum principle implies that $V^*$ has multiplicity one in $U$, and this is the key ingredient to obtain the multiplicity-one-property for $V$.

Another new difficulty in the CMC setting is that the total mass of the replacement $V^*$ may differ from the total mass of the original varifold $V$. A key observation is that the mass defect is controlled by $c\vol(U)$, which is of higher order than the mass and hence converges to zero under any blowup process. Using this insight, we are able to prove that any blowup of the min-max limit $V$ has the good replacement property of Colding-De Lellis, and is therefore regular (see Lemma \ref{L:blowup is regular}). The same observation allows us to show that the tangent cones of $V$ are always planes. 

For the regularity of $V$, the proof structure is inspired by Pitts. Namely, we first apply successive replacements $V^*$ and $V^{**}$ on two overlapping concentric annuli $A_1$ and $A_2$, with the goal of showing that they match smoothly on the overlapping region and may thus be extended all the way to the center by taking further replacements. However, in the CMC case two main issues arise: the presence of a nontrivial touching set, and the need to show that the orientations match to give the same sign for the mean curvatures. These difficulties are overcome by keeping track of the approximate replacements $\{\Om_i^*\}$ and careful analysis of their convergence behaviors (see Section \ref{SS:Compactness of stable CMC hypersurfaces}). In particular, a key step is to show that the second replacement $V^{**}$ may be represented by a boundary in $A_1\cup A_2$, so that the orientations of $V^*, V^{**}$ match as desired. Near the touching set, we use the graphical decomposition into embedded sheets together with the gluing along regular part to properly match the sheets together.

Finally, to prove that the min-max limit $V$ coincides with the extended replacement $V^*$ near the center, we face one more obstacle since in the minimal case one typically appeals to the Constancy Theorem, which does not have an analog in the CMC setting. Instead, we first directly prove the removability of the center singularity for $V^*$, and then use a moving sphere argument to show that the densities of $V$ and $V^*$ are the same in the annular region.

\subsection{Outline of the paper}

We describe our notation and basic background material in Section 1. Then in Section 2 we gather some requisite results including compactness theorems and maximum principles, in particular for almost embedded CMC hypersurfaces. 

In Section 3 we formulate the technical setup for our min-max procedure, and prove the existence of nontrivial (positive width) sweepouts (Theorem \ref{T:existence of nontrivial sweepouts}). Then for any minimizing sequence we can extract a min-max sequence $\{\partial\Om^i_{x_i}\}$ that converges in the measure-theoretic sense to a nontrivial varifold $V$.

We then present the tightening process in Section 4. In particular, we show that there exists a nice minimizing sequence $\{\Om^i_x\}$, such that every min-max sequence converges to a varifold with $c$-bounded first variation (Proposition \ref{P:tightening}). To do this, using the set of varifolds with $c$-bounded first variation as the central set, we construct a discrete gradient flow for $\Ac$, namely the {\em tightening map} (see section \ref{SS: A map from A to the space of isotopies}). Applying this map to any given minimizing sequence will result in a nice minimizing sequence by a standard contradiction argument. To deal with the issue that $\mathcal{A}^c$ is not defined for varifolds, we derive a quantitative tightening inequality (\ref{E: decrease Ac by isotopy}) inspired by Colding-De Lellis.

In Section 5, we first show that $V$ may further be taken to be $c$-almost minimizing in the sense of Definition \ref{D:c-am-varifolds} (see Theorem \ref{T: existence of almost minimizing varifold}); that is, $V$ is the limit of boundaries of so-called $(\epsilon,\delta)$-$c$-almost minimizing sets $\Omega$. The remainder of Section 5 records important properties of $c$-almost minimizing varifolds, including the existence and regularity of $c$-replacements (see Proposition \ref{P:good-replacement-property} and Lemma \ref{L:reg-replacement}). Namely, for each $(\ep, \de)$-$c$-a.m. $\Om$, one can construct a $\Ac$-minimizer $\Om^*$ among certain admissible deformations, whose boundary is a stable embedded CMC hypersurface in $U$ by classical regularity theory. We then construct the $c$-replacement $V^*$ of $V$ as the varifold limit $\lim|\partial \Om^*|$. 

Finally, in Section 6 we prove the regularity of the min-max varifold $V$ (Theorem \ref{T:main-regularity}). Namely, given two $c$-replacements $V^*$ and $V^{**}$ of $V$ in two overlapping concentric annuli, we prove that they match in $C^1$ along the boundary sphere of the smaller annulus, in order to apply the unique continuation. Using the key observation that blowups are regular and the classical maximum principle, we first show that the blowups of $V^*, V^{**}$ are identical along this sphere. The $C^1$ gluing follows readily on the regular part, after which we require a much more careful analysis to extend the gluing across the touching set. We then continue the $c$-replacement $V^*$ smoothly as an almost embedded CMC hypersurface all the way to the center of the annuli and complete the proof after removing the center singularity.

\vspace{0.5em}
{\bf Acknowledgements}: Both authors are grateful to Prof. Shing-Tung Yau for suggesting this problem and for his generous support. X. Zhou would also like to thank Prof. Richard Schoen and Prof. Neshan Wickramasekera for valuable comments. J. Zhu would also like to thank Prof. William Minicozzi for his invaluable guidance and encouragement. X. Zhou is partially supported by NSF grant DMS-1704393. J. Zhu is partially supported by NSF grant DMS-1607871.

\section{Notation}
\label{S:Notation and background}

In this section, we collect some notions. We refer to \cite{Si83} and \cite[\S 2.1]{P81} for further materials in geometric measure theory.

Let $(M^{n+1}, g)$ denote a closed, oriented, smooth Riemannian manifold of dimension $3\leq (n+1)\leq 7$. Assume that $(M, g)$ is embedded in some $\R^L$, $L\in\N$. $B_r(p), \sB_r(p)$ denote respectively the Euclidean ball of $\R^L$ or the geodesic ball of $(M, g)$. We denote by $\mH^k$ the $k$-dimensional Hausdorff measure; $\bI_{k}(M)$ the space of $k$-dimensional integral currents in $\R^L$ with support in $M$; $\Z_{k}(M)$ the space of integral currents $T\in\bI_{k}(M)$ with $\partial T=0$; $\V_{k}(M)$ the closure, in the weak topology, of the space of $k$-dimensional rectifiable varifolds in $\R^L$ with support in $M$; $G_k(M)$ the Grassmannian bundle of un-oriented $k$-planes over $M$; $\F$ and $\M$ respectively the flat norm \cite[\S 31]{Si83} and mass norm \cite[26.4]{Si83} on $\bI_k(M)$; $\mF$ the varifold $\mF$-metric on $\V_k(M)$ and currents $\mF$-metric on $\bI_k(M)$, \cite[2.1(19)(20)]{P81}; $\C(M)$ or $\C(U)$ the space of sets $\Om\subset M$ or $\Om\subset U\subset M$ with finite perimeter (Caccioppoli sets), \cite[\S 14]{Si83}\cite[\S 1.6]{Gi}; and $\X(M)$ or $\X(U)$ the space of smooth vector fields in $M$ or supported in $U$.

We also utilize the following definitions:
\begin{enumerate}[label=\alph*), leftmargin=1cm]
\label{En: notations}
\item Given $T\in\bI_{k}(M)$, $|T|$ and $\|T\|$ denote respectively the integral varifold and Radon measure in $M$ associated with $T$;
\item Given $c>0$, a varifold $V\in \V_k(M)$ is said to have {\em $c$-bounded first variation in an open subset $U\subset M$}, if
\[ |\de V(X)|\leq c \int_M|X|d\mu_V, \quad \text{for any } X\in\X(U); \]
here the first variation of $V$ along $X$ is $\de V(X)=\int_{G_k(M)} div_S X(x)d V(x, S)$, \cite[\S 39]{Si83};
\item $U_r(V)$ denotes the ball in $\V_k(M)$ under $\mF$-metric with center $V\in\V_k(M)$ and radius $r>0$;
\item Given $p\in\spt\|V\|$, $\VarTan(V,p)$ denotes the space of tangent varifolds of $V$ at $p$, \cite[42.3]{Si83};
\item Given a smooth, immersed, closed, orientable hypersurface $\Si$ in $M$, or a set $\Om\in\C(M)$ with finite perimeter, $[[\Si]]$, $[[\Om]]$ denote the corresponding integral currents with the natural orientation, and $[\Si]$, $[\Om]$ denote the corresponding integer-multiplicity varifolds;
\item $\partial\Om$ denotes the (reduced)-boundary of $[[\Om]]$ as an integral current, and $\nu_{\partial\Om}$ denotes the outward pointing unit normal of $\partial \Om$, \cite[14.2]{Si83}.
\end{enumerate}

In this paper, we are interested in the following weighted area functional defined on $\C(M)$. Given $c>0$, define the {\em $\Ac$-functional} on $\C(M)$ as
\begin{equation}
\label{E: Ac}
\Ac(\Om)=\mH^n(\partial\Om)-c\mH^{n+1}(\Om).
\end{equation}
The {\em first variation formula} for $\Ac$ along $X\in \X(M)$ is (see \cite[16.2]{Si83})
\begin{equation}
\label{E: 1st variation for Ac}
\de\Ac|_{\Om}(X)=\int_{\partial\Om}div_{\partial \Om}X d\mu_{\partial\Om}-c\int_{\partial\Om}X\cdot \nu \, d\mu_{\partial\Om},
\end{equation}
where $\nu= \nu_{\partial \Om}$ is the outward unit normal on $\partial \Om$. 

When the boundary $\partial\Om=\Si$ is a smooth immersed hypersurface, we have \[div_{\Si}X=H\,  X\cdot \nu,\] where $H$ is the mean curvature of $\Si$ with respect to $\nu$; if $\Om$ is a critical point of $\Ac$, then (\ref{E: 1st variation for Ac}) directly implies that $\Si=\partial \Om$ has constant mean curvature $c$ with respect to the outward unit normal $\nu$. In this case, we can calculate the {\em second variation formula} for $\Ac$ along normal vector fields $X\in \X(M)$ such that $X=\varphi\nu $ along $\partial\Om=\Si$ where $\varphi\in C^\infty(\Si)$, \cite[Proposition 2.5]{BCE88},
\begin{equation}
\label{E: 2nd variation for Ac}
\de^2\Ac|_{\Om}(X,X) = \Rom{2}_\Si (\varphi,\varphi) =\int_{\Si}\left( |\nabla\varphi|^2-\left(Ric^M(\nu, \nu)+|A^\Si|^2\right)\varphi^2\right)d\mu_{\Si}.
\end{equation}
In the above formula, $\nabla\varphi$ is the gradient of $\varphi$ on $\Si$; $Ric^M$ is the Ricci curvature of $M$; $A^\Si$ is the second fundamental form of $\Si$.

\section{Preliminaries}
\label{S:Preliminaries}

In this section, we collect some preliminary results. First, we study the compactness properties of stable CMC hypersurfaces. In particular, we describe the structure of the touching sets which appear naturally when one takes the limit of embedded stable CMC hypersurfaces. We also present a maximum principle for varifolds with bounded first variation, a regularity result for boundaries that minimize the $\Ac$-functional, and a result on isoperimetric profile for small volumes. 

\subsection{Compactness of stable CMC hypersurfaces}
\label{SS:Compactness of stable CMC hypersurfaces}

\begin{definition}
\label{D:stable c-hypersurface}
Let $\Si$ be a smooth, immersed, two-sided hypersurface with unit normal vector $\nu$, and $U\subset M$ an open subset. We say that $\Si$ is a {\em stable $c$-hypersurface} in $U$ if 
\begin{itemize}
\item  {the mean curvature $H$ of $\Si\cap U$ with respect to $\nu$ equals to $c$; and}
\item $\Rom{2}_\Si(\varphi, \varphi)\geq 0$ { for all $\varphi\in C^\infty(\Si)$ with $\spt{\varphi}\subset \Si\cap U$,} 
where $\Rom{2}_\Si$ is as in (\ref{E: 2nd variation for Ac}).
\end{itemize}
\end{definition}

\begin{definition}
\label{def:comparisonsheets}
Let $\Sigma_i$, $i=1,2$, be connected embedded hypersurfaces in a connected open subset $U\subset M$, with $\partial\Sigma_i\cap U=\emptyset$ and unit normals $\nu_i$. We say that {\em $\Sigma_2$ lies on one side of $\Sigma_1$} if $\Sigma_1$ divides $U$ into two connected components $U_1 \cup U_2 = U\setminus \Sigma_1$, where $\nu_1$ points into $U_1$, and either:
\begin{itemize}
\item $\Sigma_2 \subset \Clos(U_1)$, which we write as $\Sigma_1\leq \Sigma_2$ or that $\Sigma_2$ lies on the positive side of $\Sigma_1$; or
\item $\Sigma_2 \subset \Clos(U_2)$, which we write as $\Sigma_1\geq \Sigma_2$ or that $\Sigma_2$ lies on the negative side of $\Sigma_1$.
\end{itemize}
 \end{definition}

\begin{definition}[Almost embedding]
\label{D:almost embedded boundary}
Let $U\subset M^{n+1}$ be an open subset, and $\Si^n$ be a smooth $n$-dimensional manifold. A smooth immersion $\phi: \Si\rightarrow U$ is said to be an {\em almost embedding} if at any point $p\in\phi(\Si)$ where $\Si$ fails to be embedded, there is a small neighborhood $W\subset U$ of $p$, such that 
\begin{itemize}
\item $\Si\cap \phi^{-1}(W)$ is a disjoint union of connected components $\cup_{i=1}^l \Si_i$;
\item $\phi(\Si_i)$ is an embedding for each $i=1,\cdots,l$; 
\item for each $i$, any other component $\phi(\Si_j)$, $j\neq i$, lies on one side of $\phi(\Si_i)$ in $W$.
\end{itemize}
We will simply denote $\phi(\Si)$ by $\Si$ and denote $\phi(\Si_i)$ by $\Si_i$. The subset of points in $\Si$ where $\Si$ fails to be embedded will be called the {\em touching set}, and denoted by $\mS(\Si)$. We will call $\Si\backslash\mS(\Si)$ the regular set, and denote it by $\mR(\Si)$.
\end{definition}
\begin{remark}
From the definition, the collection of components $\{\Si_i\}$ meet tangentially along $\mS(\Si)$.
\end{remark}

\begin{definition}[Almost embedded $c$-boundary]
\begin{enumerate}
\item An almost embedded hypersurface $\Si\subset U$ is said to be {\em a boundary} if there is an open subset $\Om\in\C(U)$, such that $\Si$ is equal to the boundary $\partial\Om$ (in $U$) in the sense of currents;

\item The {\em outer unit normal} $\nu_\Si$ of $\Si$ is the choice of the unit normal of $\Si$ which points outside of $\Om$ along the regular part $\mR(\Si)$; 

\item $\Si$ is called {\em a stable $c$-boundary} if $\Si$ is a boundary as well as a stable immersed $c$-hypersurface. 
\end{enumerate}
\end{definition}

We have the following variant of the famous Schoen-Simon-Yau (for $2\leq n\leq 5$) \cite{SSY75} and Schoen-Simon ($n=6$) \cite{SS81} curvature estimates.
\begin{theorem}[Curvature estimates for stable $c$-hypersurfaces]
\label{T:curvature estimates and compactness}
Let $2 \leq n \leq 6$, and $U\subset M$ be an open subset. If $\Si\subset U$ is a smooth, immersed (almost embedded when $n=6$), two-sided, stable $c$-hypersurface in $U$ with $\partial \Si\cap U=\emptyset$, and $\Area(\Si)\leq C$, then there exists $C_1$ depending only on $n, M, c, C$, such that
\[ |A^\Si|^2 (x) \leq \frac{C_1}{\dist^2_M(x,\partial U)} \quad \text{ for all $x \in \Si$}. \]
Moreover if $\Si_k\subset U$ is a sequence of smooth, immersed (almost embedded when $n=6$), two-sided, stable $c$-hypersurfaces in $U$ with $\partial \Si_k\cap U=\emptyset$ and $\sup_{k} \Area(\Sigma_k) < \infty$, 
then up to a subsequence, $\Sigma_k$ converges locally smoothly (possibly with multiplicity) to some stable $c$-hypersurface $\Sigma_\infty$ in $U$. 
\end{theorem}
\begin{proof}
The compactness statement follows in the standard way from the curvature estimates. The curvature estimates follow from standard blowup arguments together with the Bernstein Theorem \cite[Theorem 2]{SSY75} and \cite[Theorem 3]{SS81}, the key being that the blowup will be a stable minimal hypersurface, and when $n=6$, the blowup of a sequence of almost embedded $c$-hypersurfaces will be embedded by the classical maximum principle for embedded minimal hypersurfaces (c.f. \cite{Colding-Minicozzi}).
\end{proof}

We need the following maximum principle.
\begin{lemma}[Maximum principle for embedded $c$-hypersurfaces]
\label{L:classical MP}
Given a connected open subset $U\subset M$, let $\Si_i\subset U$ be two connected embedded hypersurfaces with $\partial \Si_i\cap U=\emptyset$ for $i=1, 2$. Suppose that the mean curvature of each $\Si_i$ is a given constant $c>0$ with respect to the respective unit normal $\nu_i$. 
Assume that $\Sigma_2$ lies on one side of $\Sigma_1$. Then we have the following:
\begin{itemize}
\item[(i)] If there exists $p\in \Si_1\cap\Si_2$ such that $\nu_1(p)=\nu_2(p)$, then $\Si_1=\Si_2$;
\item[(ii)] Suppose $\Si_2$ lies on the negative side of $\Sigma_1$. Then if $\Sigma_1\cap\Sigma_2\neq \emptyset$, we must have $\nu_1(p)=\nu_2(p)$ for any $p\in \Si_1\cap\Si_2$, and hence $\Si_1=\Si_2$. In particular, either $\Si_1\cap\Si_2=\emptyset$ or $\Si_1=\Si_2$.
\end{itemize}
\end{lemma}
\begin{proof}
This follows directly from the classical maximum principle as follows.

Consider $p\in \Sigma_1\cap \Sigma_2$. Since $\Sigma_2$ lies on one side of $\Sigma_1$, the tangent planes must coincide at any point of their intersection. So without loss of generality we may assume that $U$ is a small ball around $p$ for which $\Sigma_1,\Sigma_2$ may be written as graphs $u_1,u_2$ in the $\nu_1$-direction over the tangent plane $T_p \Sigma_1 = T_p\Sigma_2$. 

Let $u=u_1-u_2$, then a standard computation shows that $u$ satisfies a linear elliptic equation of the form: $Lu=0$, if $\nu_2(p)=\nu_1(p)$; or $Lu=2c$ if $\nu_2(p)= -\nu_1(p)$. Here $L$ is a positive elliptic operator with smooth coefficients. Moreover,  if $\Sigma_2 \subset \Clos(U_1)$ then $u\leq 0$; if $\Sigma_2 \subset \Clos(U_2)$ then $u\geq 0$. Both items then follow from the maximum principle for nonpositive (or nonnegative) functions. 

\end{proof}

\begin{lemma}
Let $\Omega$ be a domain in $\mathbb{R}^m$ and suppose that $u$ is a classical solution on $\Omega$ of a linear inhomogenous elliptic PDE with smooth coefficients: 
\begin{equation}
Lu = a^{ij}D_{ij} u  + b^j D_j u + q u = f,  
\end{equation}
where $f$ has no zeroes on $\Omega$. Then the zero set $\{u=0\}$ is contained in a countable union of connected, embedded $(m-1)$-dimensional submanifolds. 
\end{lemma}
\begin{proof}
Let $K$ be a compact subset of $\Omega$. 

First, the implicit function theorem implies that the zero set is smooth away from the critical set. In particular, for any $\epsilon>0$ the compact set $\{u=0, |Du| \geq \epsilon\} \cap K$ is contained in the union of finitely many connected, smoothly embedded $(m-1)$-dimensional submanifolds. 

Now consider $x \in \{u=0, Du=0\}$. Then we have $a^{ij}(x) D_{ij} u(x) =f(x) \neq 0$, so by ellipticity, the Hessian $D^2 u$ must have rank at least 1. Thus for some $j$, the gradient $D(D_j u)\neq 0$, so again by the implicit function theorem there is an $r>0$ such that $B_r(x) \cap \{D_j u=0\}$ is an embedded $(m-1)$-dimensional submanifold, which clearly contains $B_r(x) \cap \{u=0\} \cap \{Du=0\}$. It follows that the compact set $\{u=0, Du=0\}\cap K$ is also contained in a finite union of connected, embedded $(m-1)$-dimensional submanifolds. 

Taking $\epsilon = 1/j\rightarrow 0$ and $K=K_j$, where $K_j$ is an exhaustion of $\Omega$, then completes the proof. 
\end{proof}

\begin{proposition}[Touching sets for almost embedded $c$-hypersurface]
\label{P:smooth touching set}
If the metric on $U^{n+1}$ is smooth, then for any almost embedded hypersurface $\Si^n \subset U$ of constant mean curvature $c$, the touching set $\mS(\Si)$ is contained in a countable union of connected, embedded $(n-1)$-dimensional submanifolds.

In particular, the regular set $\mR(\Sigma)$ is open and dense in $\Sigma$.
\end{proposition}
\begin{proof}
Let $p\in \mathcal{S}(\Sigma)$. As in the proof of Lemma \ref{L:classical MP}, there is a small neighborhood $W$ of $p$ so that the image $\Sigma \cap W$ decomposes as graphs $\{u_i\}_{i=1}^k$, ordered by height, over the common tangent plane $T_p\Sigma$. In fact by the conclusions of that lemma, after possibly shrinking $W$ there will be exactly two distinct graphs $u_1 \leq u_2$, for which the difference $u=u_1-u_2$ satisfies $Lu=2c$. The zero set of $u$ corresponds to the touching set $\mathcal{S}(\Sigma)\cap W$.  The proposition then follows from the previous lemma.
\end{proof}

\begin{remark}
In the case that the metric on $M$ is real analytic, we have the stronger statement that the touching set is a finite union of real analytic subvarieties $\bigcup_{k=0}^{n-1} S^k$ of respective dimension $k$. This follows from  \cite[Theorem 5.2.3]{KP92}, since in this setting the operator $L$ will have analytic coefficients, and hence the solution $u$ is also real analytic.
\end{remark}

\begin{theorem}[Compactness theorem for almost embedded stable $c$-hypersurfaces]
\label{T:compactness}
Let $2\leq n\leq 6$. Suppose $\Si_k\subset U$ is a sequence of smooth, almost embedded, two-sided, stable $c_k$-hypersurfaces in $U$, with $\sup_{k} \Area(\Sigma_k) < \infty$ and $\sup_k c_k <\infty$. Then the following hold:
\begin{itemize}
\item[(i)] if $\inf c_k>0$, then up to a subsequence, $\{\Sigma_k\}$ converges locally smoothly (with multiplicity) to some almost embedded stable $c$-hypersurface $\Sigma_\infty$ in $U$;
\item[(ii)] if additionally $\{\Si_k\}$ are all boundaries, then $\Sigma_\infty$ is also a boundary, and the density of $\Sigma_\infty$ is $1$ along $\mR(\Si_\infty)$ and $2$ along $\mS(\Si_\infty)$;
\item[(iii)] if $c_k\to 0$, then up to a subsequence, $\{\Sigma_k\}$ converges locally smoothly (with multiplicity) to some smooth embedded stable minimal hypersurface $\Sigma_\infty$ in $U$.
\end{itemize}
\end{theorem}
\begin{remark}
We learned that Bellettini-Wickramasekera \cite{Bellettini-Wickramasekera17} also have similar compactness results for stable CMC varifolds.
\end{remark}
\begin{proof}[Proof of Theorem \ref{T:compactness}]
Case (i) follows straightforwardly from Theorem \ref{T:curvature estimates and compactness}, the almost embedded assumption, together with the  maximum principle Lemma \ref{L:classical MP}. 

Now we prove Case (ii). Denote $\Si_k=\partial\Om_k$ for some $\Om_k\in \C(U)$. By standard compactness \cite[Theorem 6.3]{Si83}, a subsequence of $\partial\Om_k$ converges weakly as currents to some $\partial\Om_\infty$ with $\Om_\infty\in\C(U)$. We claim that $\partial\Om_\infty=\Si_\infty$ as varifold. To show this, we only need to check that the density of $\Si_\infty$ along $\mR(\Si_\infty)$ is one, and then by Lemma \ref{L:classical MP} and Proposition \ref{P:smooth touching set}, the density of $\Si_\infty$ along the touching set $\mS(\Si_\infty)$ is automatically two.

To show that the density along $\mR(\Si_\infty)$ is $1$, take an arbitrary point $p\in\mR(\Si_\infty)$. If the density at $p$ is larger than $1$, then by the locally smooth convergence of $\Si_k$ to $\Si_\infty$, there is a neighborhood $\sB_p\subset U$ of $p$, such that for $k$ large enough $\Si_k\cap \sB_p$ has a graphical decomposition as $\cup_{i=1}^{l_k}\Si_k^i$ with $l_k\geq 2$. Moreover, by Lemma \ref{L:classical MP} we have $\Si_k^1< \Si_k^2< \cdots < \Si_k^{l_k}$, and the outward unit normals $\nu_k^i$ of $\Si_k^i$ all point to the same direction. With out loss of generality, we may assume $l_k=2$ and omit the sub-index $k$. Then $\sB_p\backslash (\Si^1\cup\Si^2)$ has three connected components $U_0, U_1, U_2$ with, counting orientation, $(\partial U_0)\lc \sB_p=\Si^1$, $(\partial U_1)\lc \sB_p=\Si^2-\Si^1$, and $(\partial U_2)\lc \sB_p=-\Si^2$. 

On the other hand, for each $i$ the Constancy Theorem \cite[Theorem 26.27]{Si83} applied to $\Om_k\lc U_i$ implies that $\Om_k\lc U_i$ is identical to either $\emptyset$ or $U_i$. That is, $\Omega_k \lc \sB_p = \sum_{i=0}^2 a_i U_i$, where each $a_i =0,1$. It is then easy to see that any choice of the $a_i$ will contradict the fact that, counting orientation, $\partial(\Om_k\lc \sB_p) \lc \sB_p= \Sigma_k \cap \sB_p = \Si^1+\Si^2$.

Case (iii) follows directly from Theorem \ref{T:curvature estimates and compactness}, the almost embedded assumption, and the classical maximum principle for embedded minimal hypersurfaces (c.f. \cite{Colding-Minicozzi}).
\end{proof}

\subsection{Maximum principle for varifolds with $c$-bounded first variation}

We will need the following maximum principle which is essentially due to White \cite[Theorem 5]{White10}. 

\begin{proposition}[Maximum principle for varifolds with $c$-bounded first variation]
\label{P:maximum principle}
Suppose $V\in \V_n(M)$ has $c$-bounded first variation in a open subset $U\subset M$. Let $K\subset U$ be an open subset with compact closure in $U$, such that $\spt(\|V\|)\subset K$, and
\begin{itemize}
\item[(i)] $\partial K$ is smoothly embedded in $M$,
\item[(ii)] the mean curvature of $\partial K$ with respect to the outward pointing normal is greater than $c$.
\end{itemize}
Then $\spt(\|V\|)\cap \partial K=\emptyset$. 
\end{proposition}

\subsection{Regularity for boundaries which minimize the $\Ac$ functional}

The following result about regularity of boundaries which minimize the $\Ac$ functional can be found in \cite{Mo03}.

\begin{theorem}
\label{T:regularity of Ac minimizers}
Given $\Om\in\C(M)$, $p\in\spt\|\partial\Om\|$, and some small $r>0$, suppose that $\Om\lc \sB_r(p)$ minimizes the $\Ac$-functional: that is, for any other $\La\in\C(M)$ with $\spt\|\La-\Om\|\subset \sB_r(p)$, we have $\Ac(\La)\geq \Ac(\Om)$. Then except for a set of Hausdorff dimension at most $n-7$, $\partial\Om\lc \sB_r(p)$ is a smooth and embedded hypersurface, and is real analytic if the ambient metric on $M$ is real analytic.
\end{theorem}
\begin{proof}
Since $\Om\lc \sB_r(p)$ minimizes the $\Ac$-functional, for all $\La\in\C(M)$ as in the supposition, we have
\[ \mH^n(\partial\La)-\mH^n(\partial\Om)\geq -c|\mH^{n+1}(\La)-\mH^{n+1}(\Om)|. \]
This is precisely condition \cite[3.1(1)]{Mo03}. The regularity then follows from \cite[Corollary 3.7, 3.8]{Mo03}.
\end{proof}

\subsection{Isoperimetric profiles for small volume}

We have the following lower bound for the isoperimetric profiles for small volumes, which is a consequence of the fact that the isoperimetric profile is asymptotically Euclidean for small volumes \cite{BM82} (see also \cite[Theorem 3]{Nardulli09}). Note that although it was only stated for domains with smooth boundary, the result indeed holds for any $\Om\in\C(M)$ by using the regularity theory for isoperimetric domains (c.f. Theorem \ref{T:regularity of Ac minimizers}).
\begin{theorem}
\label{T:Isoperimetric areas}
There exists constants $C_0>0$ and $V_0>0$ depending only on $M$ such that 
\[ \Area(\partial\Om)\geq C_0 \vol(\Om)^{\frac{n}{n+1}}, \text{ whenever $\Om\in\C(M)$ and $\vol(\Om)\leq V_0$.}\]
\end{theorem}

\section{The $c$-Min-max construction}
\label{S:The c-Min-max construction}

In this section, we present the setups of the min-max construction mainly followed Pitts \cite{P81}. We also prove the existence of a non-trivial sweepout with positive $\Ac$-min-max value.

\subsection{Homotopy sequences.}\label{homotopy sequences}

We will introduce the min-max construction using the scheme developed by Almgren and Pitts \cite{AF62, AF65, P81}.

\begin{definition}\label{cell complex} 
(cell complex.)
\begin{enumerate}
\item Denote $I=[0, 1]$, $I_0=\partial I=I\backslash (0, 1)$;

\item For $j\in\N$, $I(1, j)$ is the cell complex of $I$, whose $1$-cells are all intervals of form $[\frac{i}{3^{j}}, \frac{i+1}{3^{j}}]$, and $0$-cells are all points $[\frac{i}{3^{j}}]$; 

\item For $p=0 ,1$, $\al\in I(1, j)$ is a $p$-cell if $dim(\al)=p$. $0$-cell is also called a vertex;

\item $I(1, j)_p$ denotes the set of all $p$-cells in $I(1, j)$, and $I_0(1, j)_0$ denotes the set $\{[0], [1]\}$;

\item Given a $1$-cell $\al\in I(1, j)_1$, and $k\in\N$, $\al(k)$ denotes the $1$-dimensional sub-complex of $I(1, j+k)$ formed by all cells contained in $\al$. For $q=0, 1$, $\al(k)_q$ and $\al_0(k)_q$ denote respectively the set of all $q$-cells of $I(1, j+k)$ contained in $\al$, or in the boundary of $\al$;

\item The boundary homeomorphism $\partial: I(1, j)\rightarrow I(1, j)$ is given by $\partial[a, b]=[b]-[a]$ if $[a, b]\in I(1, j)_1$, and $\partial[a]=0$ if $[a]\in I(1, j)_0$;

\item The distance function $\md: I(1, j)_0\times I(1, j)_0\rightarrow\N$ is defined as $\md(x, y)=3^{j}|x-y|$;

\item The map $\n(i, j): I(1, i)_{0}\to I(1, j)_{0}$ is defined as: $\n(i, j)(x)\in I(1, j)_{0}$ is the unique element of $I(1, j)_0$, such that $\md\big(x, \n(i, j)(x)\big)=\inf\big\{\md(x, y): y\in I(1, j)_{0}\big\}$.
\end{enumerate}
\end{definition}

Consider a map to the space of Caccioppoli sets: $\phi: I(1, j)_{0}\rightarrow\C(M)$. The \emph{fineness} of $\phi$ is defined as:
\begin{equation}\label{fineness}
\mf(\phi)=\sup\Big\{\frac{\M\big(\partial\phi(x)-\partial\phi(y)\big)}{\md(x, y)}:\ x, y\in I(1, j)_{0}, x\neq y\Big\}.
\end{equation}
Similarly we can define the fineness of $\phi$ with respect to the $\F$-norm and $\mF$-metric. We use $\phi: I(1, j)_{0}\rightarrow\big(\C(M), \{0\}\big)$ to denote a map such that $\phi\big(I(1, j)_{0}\big)\subset\C(M)$ and $\partial\phi|_{I_{0}(1, j)_{0}}=0$, i.e. $\phi([0]), \phi([1])=\emptyset$ or $M$.

\begin{definition}\label{homotopy for maps}
Given $\de>0$ and $\phi_{i}: I(1, k_{i})_{0}\rightarrow\big(\C(M), \{0\}\big)$, $i=0,1$, we say \emph{$\phi_{1}$ is $1$-homotopic to $\phi_{2}$ in $\big(\C(M), \{0\}\big)$ with fineness $\de$}, if $\exists\ k_{3}\in\N$, $k_{3}\geq\max\{k_{1}, k_{2}\}$, and
$$\psi: I(1, k_{3})_{0}\times I(1, k_{3})_{0}\rightarrow \C(M),$$
such that
\begin{itemize}
\setlength{\itemindent}{1em}
\item $\mf(\psi)\leq \de$;
\item $\psi([i], x)=\phi_{i}\big(\n(k_{3}, k_{i})(x)\big)$, $i=0,1$;
\item $\partial\psi\big(I(1, k_{3})_{0}\times I_{0}(1, k_{3})_{0}\big)=0$.
\end{itemize}
\end{definition}

\begin{definition}\label{(1, M) homotopy sequence}
A \emph{$(1, \M)$-homotopy sequence of mappings into $\big(\C(M), \{0\}\big)$} is a sequence of mappings $\{\phi_{i}\}_{i\in\N}$,
$$\phi_{i}: I(1, k_{i})_{0}\rightarrow\big(\C(M), \{0\}\big),$$
such that $\phi_{i}$ is $1$-homotopic to $\phi_{i+1}$ in $\big(\C(M), \{0\}\big)$ with fineness $\de_{i}$, and
\begin{itemize}
\setlength{\itemindent}{1em}
\item $\lim_{i\rightarrow\infty}\de_{i}=0$;
\item $\sup_{i}\big\{\M(\partial\phi_{i}(x)):\ x\in I(1, k_{i})_{0}\big\}<+\infty$.
\end{itemize}
\end{definition}
\begin{remark}
Note that the second condition implies that $\sup_{i}\big\{\Ac(\phi_{i}(x)):\ x\in I(1, k_{i})_{0}\big\}<+\infty$.
\end{remark}

\begin{definition}\label{homotopy for sequences}
Given two $(1, \M)$-homotopy sequences of mappings $S_{1}=\{\phi^{1}_{i}\}_{i\in\N}$ and $S_{2}=\{\phi^{2}_{i}\}_{i\in\N}$ into $\big(\C(M), \{0\}\big)$, \emph{$S_{1}$ is homotopic to $S_{2}$} if $\exists\ \{\de_{i}\}_{i\in\N}$, such that
\begin{itemize}
\setlength{\itemindent}{1em}
\item $\phi^{1}_{i}$ is $1$-homotopic to $\phi^{2}_{i}$ in $\big(\C(M), \{0\}\big)$ with fineness $\de_{i}$;
\item $\lim_{i\rightarrow \infty}\de_{i}=0$.
\end{itemize}
\end{definition}

It is easy to see that the relation ``is homotopic to" is an equivalence relation on the space of $(1, \M)$-homotopy sequences of mappings into $\big(\C(M), \{0\}\big)$. An equivalence class is a \emph{$(1, \M)$-homotopy class of mappings into $\big(\C(M), \{0\}\big)$}. Denote the set of all equivalence classes by $\pi^{\#}_{1}\big(\C(M, \M), \{0\}\big)$.

\subsection{Min-max construction.}

\begin{definition}
(Min-max definition) Given $\Pi\in\pi^{\#}_{1}\big(\C(M, \M), \{0\}\big)$, define: $\bL^c: \Pi\rightarrow\R^{+}$ as a function given by:
\[ \bL^c(S)=\bL^c(\{\phi_{i}\}_{i\in\N})=\limsup_{i\rightarrow\infty}\max\big\{\Ac\big(\phi_{i}(x)\big):\ x \textrm{ lies in the domain of $\phi_{i}$}\big\}. \]
The \emph{$\mathcal{A}^c$-min-max value of $\Pi$} is defined as
\begin{equation}\label{width}
\bL^c(\Pi)=\inf\{\bL^c(S):\ S\in\Pi\}.
\end{equation}

A sequence $S=\{\phi_i\}\in\Pi$ is called a \emph{critical sequence} if $\bL^c(S)=\bL^c(\Pi)$. 

Given a critical sequence $S$, then $K(S)=\{V=\lim_{j\to \infty}|\partial \phi_{i_{j}}(x_{j})|:\ \textrm{$x_{j}$ lies in the domain of $\phi_{i_{j}}$}\}$ is a compact subsets of $\V_n(M^{n+1})$. 
The \emph{critical set} of $S$ is the subset $C(S)\subset K(S)$ defined by 
\[ C(S)=\{V=\lim_{j\rightarrow\infty}|\partial\phi_{i_j}(x_j)|:\ \text{with} \lim_{j\to\infty} \Ac(\phi_{i_j}(x_j))=\bL^c(S)\}. \]
\end{definition}


Note that by \cite[4.1(4)]{P81}, we immediately have:
\begin{lemma}
\label{L:critical sequence}
Given any $\Pi\in\pi^{\#}_{1}\big(\C(M, \M), \{0\}\big)$, there exists a critical sequence $S\in \Pi$.
\end{lemma}

The main theorem of this paper is as follows:
\begin{theorem}
\label{thm:main-body}
Let $2\leq n\leq 6$. Given a smooth closed Riemannian manifold $M^{n+1}$ and $c>0$, there exists $\Pi\in\pi^{\#}_{1}\big(\C(M, \M), \{0\}\big)$ and a critical sequence $S\in \Pi$ such that:
\begin{itemize}
\item $\bL^c(\Pi)=\bL^c(S)>0$;
\item There exists an element of $C(S)$ induced by a nontrivial, smooth, almost embedded, closed hypersurface $\Sigma^n \subset M$ of constant mean curvature $c$ with multiplicity one.
\end{itemize}
\end{theorem}

\begin{proof}[Proof of Theorem \ref{thm:main-body}]
This follows from combining Theorem \ref{T:existence of nontrivial sweepouts}, Theorem \ref{T: existence of almost minimizing varifold} and Theorem \ref{T:main-regularity}.
\end{proof}

\subsection{Existence of nontrivial sweepouts}
\label{SS:Existence of nontrivial sweepouts}

\begin{theorem}
\label{T:existence of nontrivial sweepouts}
There exists $\Pi \in \pi^{\sharp}_1\big(\C(M, \M), \{0\}\big)$, such that for any $c>0$, we have $\bL^c(\Pi)>0$.
\end{theorem}
\begin{remark}
\label{rmk:lower bound of Ac}
Let us first describe a heuristic argument using smooth sweepouts which will help to reveal the key idea. Let $C_0>0$ and $V_0>0$ to be the constants in Theorem \ref{T:Isoperimetric areas}, and fix $0<V\leq V_0$ such that $V^{\frac{-1}{n+1}} >2c/C_0$. Note that $V$ only depends on $c,C_0,V_0$. 

Then for any $\Om$ with $\vol(\Om)=V$, we have
\begin{equation}
\label{E:lower bound of Ac}
\Ac(\Om) \geq C_0 V^{\frac{n}{n+1}} - cV > cV >0.
\end{equation}

Now consider any smooth 1-parameter family $\{\Om_x: x\in [0, 1]\}$ satisfying $\Om_0=\emptyset$ and $\Om_1=M$. Since $\{\Om_x\}$ sweeps out $M$, there must exist some $x_0\in(0, 1)$ such that $\vol(\Om_{x_0})=V$, whence $\max_{x\in [0, 1]}\Ac(\Om_x) \geq cV >0$. Since this holds for any sweepout we then have $\bL^c(\Pi) \geq cV>0$. 
\end{remark}

\begin{proof}[Proof of Theorem \ref{T:existence of nontrivial sweepouts}]
Take a Morse function $\phi: M\to [0, 1]$, and consider the sub-level sets $\Phi: [0, 1]\to \C(M)$, given by $\Phi(t)=\{x\in M: \phi(x)<t\}$. By the interpolation theorem of the first author \cite[Theorem 5.1]{Zhou15b}, $\Phi$ can be discretized to a $(1, \M)$-homotopy sequence $\overline{S}=\{\overline{\phi}_i\}$ where $\overline{\phi}_i: I(1, k_i)_0\to (\C(M), \{0\})$. Moreover, under Almgren's isomorphism $F_A$ \cite[\S 3.2]{AF62} (see also \cite[\S 4.2]{Zhou15b}), $\overline{\phi}_i$ is mapped to the fundamental class in $H_{n+1}(M)$ for $i$ large, i.e. $F_A(\overline{\phi}_i)=[[M]]$. Consider $\Pi=[\overline{S}]$, then by \cite[Theorem 7.1]{AF62} for any $S=\{\phi_i\}\in\Pi$, we have $F_A(\phi_i)=[[M]]$ for $i$ large. In particular, this means that $\sum_{j\in I(1, k_i)_0}Q_j=M$ as currents; here for given $j$ and $\al_j=[\frac{j-1}{3^{k_i}}, \frac{j}{3^{k_i}}]=[x_{j-1}, x_j]$, $Q_j\in \bI_{n+1}(M)$ is the isoperimetric choice (c.f. \cite[1.14]{AF62}) of $\partial\phi_i(x_j)-\partial\phi_i(x_{j-1})$, i.e. 
\[ \M(Q_j)= \F(\partial\phi_i(x_j)-\partial\phi_i(x_{j-1})), \text{ and } \partial Q_\al=\partial\phi_i(x_j)-\partial\phi_i(x_{j-1}). \]
Denote $\Om_l=\sum_{j=1}^lQ_j$. Then we have
\[ \partial \Om_l = \partial \phi_i(x_l)-\partial \phi_i(0)=\partial \phi_i(x_l), \]
and this implies, by the Constancy Theorem \cite[26.27]{Si83}, that $\Om_l$ is a Caccioppoli set (possibly with a negative orientation) when $\M(\Om_l)<\vol(M)$. Note that $\M(Q_j)<\mf(\phi_i)=\de_i$, hence by continuity there exists some $l_c\in \N$, $l_c<3^{k_i}$, such that $\M(\Om_{l_c})\in [V-\de_i, V+\de_i]$, where $V$ is as in Remark \ref{rmk:lower bound of Ac}, and $\phi_i(x_{l_c})=\Om_{l_c}$. Then the same argument as in the remark above gives a uniform positive lower bound for $\Ac(\phi_i(x_{l_c}))$, and this finishes the proof.
\end{proof}

\section{Tightening}
\label{S:tightening}

In this section, we construct the tightening map adapted to the $\Ac$ functional and prove that after applying the tightening map to a critical sequence, every element in the critical set has uniformly bounded first variation. Our approach is adapted from those in \cite[\S 4]{CD03} and \cite[\S 4.3]{P81}.

\subsection{Annular decomposition}

Given $L>0$, consider the set of varifolds in $\V_n(M)$ with $2L$-bounded mass: $A^L=\{V\in \V_n(M):\ \|V\|(M)\leq 2L\}$. 
Denote
$$A^c_\infty=\big\{V\in A^L: |\de V(X)|\leq c \int|X| d\mu_V, \text{ for any }X\in\X(M) \big\}.$$
Consider the concentric annuli around $A_\infty$ under the $\mF$-metric, i.e.
$$A_j=\big\{V\in A^L: \frac{1}{2^j}\leq \mF(V, A^c_\infty)\leq \frac{1}{2^{j-1}}\big\}, \quad j\in\N.$$

Since $c$-bounded first variation is a closed condition, we have 
\begin{lemma}\label{A0 is compact}
$A^c_\infty$ is a compact subset of $A^L$ under the $\mF$-metric.
\end{lemma}
 
It is easy to show (by contradiction, for instance) that for any varifold in $A_j$, we can find a vector field satisfying the following condition.
\begin{lemma}
\label{L:vf to deform Ac}
For any $V\in A_j$, there exists $X_V\in\X(M)$, such that
\begin{equation}
\label{E:1st variation upper bounded}
\|X_V\|_{C^1(M)}\leq 1, \quad \de V(X_V)-c\int_M |X_V|d\mu_V\leq -c_j<0,
\end{equation}
where $c_j$ depends only on $j$.
\end{lemma}

\subsection{A map from $A^L$ to the space of vector fields}

In this part, we will construct a map $X: A^L \rightarrow \X(M)$, which is continuous with respect to the $C^1$ topology on $\X(M)$.

Given $V\in A_j$, let $X_V$ be given in Lemma \ref{L:vf to deform Ac}. Since $div_S X_V$ is Lipschitz on $G_n(M)$ for fixed $X_V$, the map
\[ W \to \de W(X_V)-c \int_M|X_V|d\mu_W=\int_{G_n(M)}div_S (X_V) dW(x, S)-c \int_M|X_V|d\mu_W \]
is continuous with respect to the $\mF$-metric. Therefore for any $V\in A_j$, there exists $0<r_V<\frac{1}{2^{j+1}}$, such that for any $W\in U_{r_V}(V)$, i.e. $\mF(W, V)<r_V$,
\begin{equation}
\label{E:1st variation upper bound 2}
\de W(X_V)-c \int |X_V|d\mu_W \leq \frac{1}{2}\left(\de V(X_V)- c \int |X_V|d\mu_V\right)\leq -\frac{1}{2}c_j<0.
\end{equation}
Now $\big\{U_{r_V/2}(V): V\in A_j\big\}$ is an open covering of $A_j$. By the compactness of $A_j$, we can find finitely many balls $\big\{U_{r_{j, i}}(V_{j, i}): V_{j, i}\in A_j, 1\leq i\leq q_j\big\}$, where $r_{j, i}=r_{V_{j, i}}$, such that
\begin{itemize}

\item[$(\rom{1})$] The balls $U_{r_{j, i}/2}(V_{j, i})$ with half radii cover $A_j$;

\item[$(\rom{2})$] The balls $U_{r_{j, i}}(V_{j, i})$ are disjoint from $A_k$ if $|j-k|\geq 2$.
\end{itemize}
In the following, we denote $U_{r_{j, i}}(V_{j, i})$, $U_{r_{j, i}/2}(V_{j, i})$ and $X_{V_{j, i}}$ by $U_{j, i}$, $\ti{U}_{j, i}$ and $X_{j, i}$ respectively.

Now we can construct a partition of unity $\{\varphi_{j, i}: j\in\N, 1\leq i\leq q_j\}$ sub-coordinate to the covering $\big\{\ti{U}_{j, i}, 1\leq i\leq q_j, j\in\N\big\}$ by
$$\varphi_{j, i}(V)=\frac{\psi_{j, i}(V)}{\sum\{\psi_{p, q}(V), p\in\N, 1\leq q\leq q_p\}},$$
where $\psi_{j, i}(V)=\mF(V, A^L\backslash \ti{U}_{j, i})$.

The map $X: A^L\rightarrow \X(M)$ is defined by
\begin{equation}
\label{E:construction of vector fields H}
X(V)=\mF(V, A^c_\infty) \sum_{j\in\N, 1\leq i\leq q_j}\varphi_{j, i}(V)X_{j, i}.
\end{equation}
The following lemma is a straightforward consequence of the construction.
\begin{lemma}
\label{L:properties of HV}
The map $X: V\rightarrow X(V)$ is continuous with respect to the $C^1$ topology on $\X(M)$.
\end{lemma}


\subsection{A map from $A^L$ to the space of isotopies}
\label{SS: A map from A to the space of isotopies}

In this part, we will associate each $V\in A$ with an isotopy of $M$ in a continuous manner. The isotopy will be generated by the vector field $X(V)$. In particular, given $V\in A^L$, we use $\Phi_V: \R^+\times M\rightarrow M$ to denote the one parameter group of diffeomorphisms generated by $X(V)$. 

Given $\Om\in\C(M)$ with $\partial\Om\in A^L$, we will deform $\Om$ by $\Phi_{|\partial \Om|}(t)$ to get a 1-parameter family of sets of finite perimeter $\Om_t=\Phi_{|\partial\Om|}(t)(\Om)$, and we will show that the $\Ac$ functional of $\Om_t$ for some $t>0$ can be deformed down by a fixed amount depending only on $\mF(|\partial\Om|, A^c_\infty)$.

In fact, given $V\in A_j$, let $\rho(V)$ be the smallest radii of the balls $\ti{U}_{k, i}$ which contain $V$. As there are only finitely many balls $\ti{U}_{k, i}$ which intersect $A_j$ nontrivially, we know that $\rho(V)\geq r_j>0$, where $r_j$ depends only on $j$; moreover, by construction the sub-index $k$ of these $\ti{U}_{k, i}$ can only be $j-1, j$, or $j+1$. Then by (\ref{E:1st variation upper bound 2}) and (\ref{E:construction of vector fields H}), we have for any $W\in U_{\rho(V)}(V)$ that
$$\de W(X(V))-c\int |X(V)|d\mu_W \leq -\frac{1}{2^{j+1}}\min\{c_{j-1}, c_j, c_{j+1}\}.$$
Therefore we can find two continuous functions $g:\R^+\rightarrow \R^+$ and $\rho: \R^+\rightarrow \R^+$, such that $\rho(0)=0$ and 
\[ \de W(X(V))-c\int |X(V)|d\mu_W \leq -g\big(\mF(V, A^c_\infty)\big),\quad \text{ if } \mF(W, V)\leq \rho\big(\mF(V, A^c_\infty)\big). \]
In particular,  by (\ref{E: 1st variation for Ac}),
\begin{equation}
\label{E:1st variation upper bound-continuous version}
\de \Ac|_\Om(X(V))\leq -g\big(\mF(V, A^c_\infty)\big),\quad \text{ if } \Om\in\C(M), \mF(|\partial\Om|, V)\leq \rho\big(\mF(V, A^c_\infty)\big).
\end{equation}

Next, we will construct a continuous time function $T: [0, \infty) \rightarrow [0, \infty)$, such that
\begin{itemize}
\item[$(\rom{1})$] $\lim_{t\rightarrow 0}T(t)=0$, and $T(t)>0$ if $t\neq 0$;

\item[$(\rom{2})$] For any $V\in A^L$, denote $\ga=\mF(V, A^c_\infty)$; then $V_t=\big(\Phi_V(t)\big)_{\#}V\in U_{\rho(\ga)}(V)$ for all $0\leq t\leq T(\ga)$.
\end{itemize}
In fact, given $V\in A_j$, and $\rho=\rho\big(\mF(V, A^c_\infty)\big)>0$, there exists $T_V>0$, such that $V_t\in U_\rho(V)$ for all $0\leq t\leq T_V$. Moreover, by the compactness of $A_j$ and the continuity of $\Phi_V(t)_\# V$ in $V$ and $t$, we may choose $T_V$ such that $T_V\geq T_j>0$ for all $V\in A_j$, where $T_j$ depends only on $j$. Interpolating between the $T_j$ yields the desired continuous function $T$ depending only on $\mF(V, A^c_\infty)$. 

\vspace{1em}
In summary, given $V\in A^L\backslash A^c_\infty$, denote $\ga=\mF(V, A^c_\infty)>0$,
\[ \Psi_V(t, \cdot)=\Phi_V\big(T(\ga)t, \cdot\big),\quad \textrm{for } t\in[0, 1], \]
and $L: \R^+\rightarrow \R^+$, with $L(\ga)=T(\ga)g(\ga)$; then $L(0)=0$ and $L(\ga)>0$ if $\ga>0$. We can deform $V$ through a continuous family  $\left\{V_t=\big(\Psi_V(t)\big)_{\#}V: t\in [0, 1]\right\} \subset U_{\rho(\gamma)}(V)$, such that
\begin{itemize}
\item[$(\rom{1})$] The map $(t, V)\rightarrow V_t$ is continuous under the $\mF$-metric;

\item[$(\rom{2})$] Using (\ref{E:1st variation upper bound-continuous version}), when $V=\partial \Om$, $\Om\in\C(M)$, $\ga=\mF(|\partial\Om|, A^c_\infty)>0$, we have
\begin{equation}
\label{E: decrease Ac by isotopy}
\begin{split}
\Ac(\Om_1)-\Ac(\Om) & \leq \int_{0}^{T(\ga)}[\de \Ac|_{\Om_t}] (X(|\partial\Om)|) dt\leq -T(\ga)g(\ga)\\
                                       & = -L(\ga)<0.
\end{split}
\end{equation}
\end{itemize}
Finally note that the flow $\Psi_V(t, \cdot)$ is generated by the vector field 
\begin{equation}
\label{E:tiX}
\ti{X}(V)=T(\ga)X(V).
\end{equation}

\subsection{Deforming sweepouts by the tightening map}
\label{SS:Deforming sweepouts by the tightening map}

Applying our tightening map constructed above in place of \cite[\S 4.3]{P81} to a critical sequence provided by Lemma \ref{L:critical sequence}, we can deduce the following result. 

\begin{proposition}[Tightening]
\label{P:tightening}
Let $\Pi \in \pi^{\sharp}_1\big(\C(M, \M), \{0\}\big)$, and assume $\bL^c(\Pi)>0$. For any critical sequence $S^*$ for $\Pi$, there exists another critical sequence $S$ for $\Pi$ such that $C(S) \subset C(S^*)$ and each $V\in C(S)$ has $c$-bounded first variation.
\end{proposition}

\begin{proof}
Take $S^*=\{\phi_i^*\}$, where $\phi^*_i: I(1, k_i)_{0}\to\big(\C(M), \{0\}\big)$, and $\phi^*_i$ is $1$-homotopic to $\phi^*_{i+1}$ in $\big(\C(M), \{0\}\big)$ with fineness $\de_i \searrow 0$. Let $\Xi_i: I(1, k_i)_0\times [0, 1]\to \C(M)$ be defined as
\[ \Xi_i(x, t)= \Psi_{|\partial \phi^*_i(x)|}(t) \big( \phi^*_i(x)\big) . \]

Denote $\phi_i^t(\cdot)=\Xi_i(\cdot, t)$. Heuristically, one would like to set $\phi_i=\phi_i^1$ as the desired sequence, but since the isotopies $\Psi_{|\partial \phi^*_i(x)|}$ depend on $x$, the fineness of $\{\phi^1_i\}$ could be large even if $\mf(\phi^*_i)$ is small. Thus we need to interpolate $\phi^1_i$ to get the desired $\phi_i$, but we need to make sure the values of $\phi_i$ after interpolation are $\mF$-close to those of $\phi^1_i$. Similar difficulties appeared in the same way in \cite[\S 15]{MN14}. The authors in \cite{MN14} used a discrete-to-continuous interpolation argument. Unfortunately we cannot adapt their argument, since their constructions involve currents which may not be boundaries of Caccioppoli sets. Instead, we develop another interpolation method in Claim 2. Before that, we pause to prove:

\vspace{0.3cm}
\textit{Claim 1: if $\lim_{i\to\infty}\Ac(\phi^1_i(x_i))=\bL^c(\Pi)$, then (up to relabeling) there is a subsequence $\{\phi^1_i(x_i)\}$ converging (as varifolds) to a varifold in $C(S^*)$ of $c$-bounded first variation.}

\vspace{0.3cm}
\textit{Proof of Claim 1:} By (\ref{E: decrease Ac by isotopy}),
\begin{equation}
\Ac(\phi_i^1(x_i))-\Ac(\phi^*_i(x_i)) = -L(\ga_i),
\end{equation}
where $\ga_i=\mF(|\partial \phi^*_i(x_i)|, A^c_\infty)$. Therefore,
\[ \bL^c(\Pi) = \lim \Ac(\phi_i^1(x_i)) = \lim \Ac(\phi^*_i(x_i)) - L(\lim \ga_i) \leq \bL^c(\Pi) - L(\lim\ga_i), \] 
so actually we must have $\lim \ga_i=0$ and this implies that $\lim|\partial \phi^*_i(x_i)| \in A^c_\infty$. Moreover, by our construction of the tightening map, each $|\partial\phi^1_i(x_i)|$ had to be $\rho(\ga_i)$-close to $|\partial\phi^*_i(x_i)|$ under the $\mF$-metric, therefore \[\lim|\partial \phi^1_i(x_i)|=\lim|\partial \phi^*_i(x_i)| \in A^c_\infty\cap C(S^*),\] and this finishes the proof of the claim.

\vspace{0.3cm}
\textit{Claim 2: there exist integers $l_i>k_i$ and maps $\phi_i: I(1, l_i)_0\to (\C(M), \{0\})$ for each $i$, such that $S=\{\phi_i\}$ is homotopic to $S^*$, and
\begin{itemize}
\item[(a)] $\phi_i^1=\phi_i\circ \n(l_i, k_i)$ on $I(1, k_i)_0$;
\item[(b)] $\mf(\phi_i)\to 0$, as $i\to \infty$;
\item[(c)]  $\Ac(\phi_i(x))-\max\{\Ac(\phi^1_i(y)): \al\in I(1, k_i)_1, x, y\in \al\}\to 0$, uniformly in $x\in I(1, l_i)_0$ as $i\rightarrow \infty$. 
\item[(d)] $\max\{\mF(\partial\phi_i(x), \partial \phi^1_i(y)): \al\in I(1, k_i)_1, x, y\in \al\} \to 0$, as $i\rightarrow \infty$.
\end{itemize}}

\textit{Proof of Claim 2:} 
The idea is to extend $\phi^1_i$ to a \textit{piecewise} continuous (with respect to the $\mF$-metric) map on $I$ and then apply the discretization result in \cite[Theorem 5.1]{Zhou15b}. However, since this procedure is somewhat technical, the proof is deferred to Appendix \ref{A:proof of claim 2}. \qed

\vspace{0.3cm}
In particular, $S$ is a valid sequence in $\Pi$, and we now check that it satisfies the requirements of the proposition. First, property (c) and the fact that $S^*$ is a critical sequence directly imply that $S$ is also a critical sequence. It remains to show that every element in $C(S)$ must lie in $C(S^*)$ and have $c$-bounded first variation. Given $V\in C(S)$, one can find a subsequence (without relabeling) $\{\phi_i(\overline{x}_i): \overline{x}_i\in I(1, l_i)_0\}\subset \C(M)$, such that $V=\lim |\partial \phi_i(\overline{x}_i)|$ as varifolds, and
\[ \lim\Ac(\phi_i(\overline{x}_i))=\bL^c(\Pi). \]
We will need to first consider $\phi_i(x_i)=\phi_i^1(x_i)$, where $x_i$ is the nearest point to $\overline{x}_i$ in $I(1, k_i)_0$. By (c) and (d), we have $\lim\Ac(\phi^1_i(x_i))=\bL^c(\Pi)$ and also $\lim |\partial \phi_i(\overline{x}_i)|=\lim |\partial \phi^1_i(x_i)|$ as varifolds. Then by Claim 1, we conclude that $V\in A^c_\infty\cap C(S^*)$. This completes the proof. 
\end{proof}

\section{$c$-Almost minimizing}
\label{S:c-Almost minimizing}

In this section, we introduce the notion of $c$-almost minimizing varifolds, and prove the existence of such a varifold from min-max construction. We prove the existence of a $c$-replacement for any $c$-almost minimizing varifold. Using this property, we show that every blowup of such varifold is regular. As an easy consequence, the tangent cones of such varifolds are always integer multiples of planes. 

\begin{definition}[$c$-almost minimizing varifolds]
\label{D:c-am-varifolds}
Let $\nu$ be the $\F$, $\M$-norms or the $\mF$-metric. For any given $\ep, \de>0$ and an open subset $U\subset M$, we define $\sA^c_n(U; \ep, \de; \nu)$ to be the set of all $\Om\in\C(M)$ such that if $\Om=\Om_0, \Om_1, \Om_2, \cdots, \Om_m\in\C(M)$ is a sequence with:
\begin{itemize}
\item[(i)] $\spt(\Om_i-\Om)\subset U$;
\item[(ii)] $\nu(\partial\Om_{i+1}, \partial\Om_i)\leq \de$;
\item[(iii)] $\Ac(\Om_i)\leq \Ac(\Om)+\de$, for $i=1, \cdots, m$,
\end{itemize}
then $\Ac(\Om_m)\geq \Ac(\Om)-\ep$.

We say that a varifold $V\in\V_n(M)$ is {\em $c$-almost minimizing in $U$} if there exist sequences $\ep_i \to 0$, $\de_i \to 0$, and $\Om_i\in \sA^c_n(U; \ep_i, \de_i; \F)$, such that $\mF(|\partial\Om_i|, V)\leq \ep_i$.
\end{definition}

The following simple fact says that $c$-almost minimizing implies $c$-bounded first variation.
\begin{lemma}
\label{L:c-am implies c-bd-first-variation}
Let $V\in\V_n(M)$ be $c$-almost minimizing in $U$, then $V$ has $c$-bounded first variation in $U$.
\end{lemma}
\begin{proof}
Suppose by contradiction that $V$ does not have $c$-bounded first variation, then there exist $\ep_0>0$ and a smooth vector field $X\in\X(U)$ compactly supported in $U$, such that
\[ \left|\int_{G_n(M)} div_S X(x) dV(x, S)\right|\geq (c+\ep_0)\int_M|X|d\mu_V>0. \]
By changing the sign of $X$ if necessary, we have
\[ \int_{G_n(M)} div_S X(x) dV(x, S)\leq -(c+\ep_0)\int_M |X|d\mu_V. \]
By continuity, we can find $\ep_1>0$ small enough depending only on $\ep_0, V, X$, such that if $\Om\in \C(M)$ with $\mF(|\partial\Om|, V)<2\ep_1$, then 
\[ \de\Ac|_\Om(X) \leq \int_{\partial\Om} div_{\partial\Om} X d\mu_{\partial\Om} + c\int_{\partial\Om} |X|d\mu_{\partial\Om}\leq -\frac{\ep_0}{2}\int_{M} |X| d\mu_{V} <0. \]

If $\mF(|\partial\Om|, V)<\ep_1$, then by deforming $\Om$ along the $1$-parameter flow $\{\Phi^X(t): t\in[0, \tau)\}$ of $X$ for a uniform short time $\tau>0$, we can obtain a $1$-parameter family $\{\Om_t\in\C(M): t\in[0, \tau)\}$, such that $t\to \partial\Om_t$ is continuous under the $\F$-topology, with $\spt(\Om_t-\Om)\subset U$, $\mF(|\partial\Om_t|, V)<2\ep_1$ and $\Ac(\Om_t)\leq \Ac(\Om_0)=\Ac(\Om)$ for all $t\in[0, \tau)$, but with $\Ac(\Om_{\tau})\leq \Ac(\Om)-\ep_2$ for some $\ep_2>0$ depending only on $\ep_0, \ep_1, V, X$. 

Summarizing the above, given any $\ep<\min\{\ep_1, \ep_2\}$ and $\de>0$, if $\Om\in\C(M)$ and $\mF(|\partial\Om|, V)<\ep$, then $\Om\notin \sA^c_n(U; \ep, \de; \F)$; this contradicts the $c$-almost minimizing property of $V$.
\end{proof}

We will need the following equivalence result among several almost minimizing concepts using the three different topology. In particular, we can actually use the $\M$-norm instead at the expense of shrinking the open subset $U \subset M$.
\begin{proposition}
\label{P:def-equiv}
Given $V\in \V_n(M)$, then the following statements satisfy $(a)\Longrightarrow (b)\Longrightarrow (c)\Longrightarrow (d)$:
\begin{itemize}
\item[$(a)$]  $V$ is $c$-almost minimizing in $U$;
\item[$(b)$]  For any $\ep>0$, there exists $\de>0$ and $\Om \, \in \, \sA^c_n(U; \ep, \de; \mF)$ such that $\mF(V, |\partial \Om|)<\ep$;
\item[$(c)$]  For any $\ep>0$, there exists $\de>0$ and $\Om \in \sA^c_n(U; \ep, \de; \M)$ such that $\mF(V, |\partial\Om|)<\ep$;
\item[$(d)$] $V$ is $c$-almost minimizing in $W$ for any relatively open subset $W \subset \subset U$.
\end{itemize}
\end{proposition}
\begin{remark}
The proof was originally due to Pitts \cite[Theorem 3.9]{P81}. In our context, we work with boundaries instead of general integral currents. Furthermore, in Definition \ref{D:c-am-varifolds}(iii), we use the $\Ac$ functional instead of the mass $\M$. 
\end{remark}
\begin{proof}
It is easy to see $(a)\Longrightarrow (b)\Longrightarrow (c)$. The last implication $(c)\Longrightarrow (d)$ is an interpolation process which was originally established in Pitts \cite[Proposition 3.8]{P81} using integral cycles. The corresponding interpolation process using boundaries of Caccioppoli sets was obtained by the first author in \cite[Proposition 5.3]{Zhou15b}. The detailed description of this process is given in Lemma \ref{L:interpolation1} in Appendix \ref{A:An interpolation lemma}.
\end{proof}

\begin{definition}
\label{D:am-annuli}
A varifold $V \in \V_n(M)$ is said to be \emph{$c$-almost minimizing in small annuli} if for each $p\in M$, there exists $r_{am}(p) >0$ such that $V$ is $c$-almost minimizing in $A_{s, r}(p)\cap M$ for all $0<s<r\leq r_{am}(p)$, where $A_{s, r}(p)=B_r(p)\backslash B_s(p)$.
\end{definition}

\begin{theorem}[Existence of $c$-almost minimizing varifold] 
\label{T: existence of almost minimizing varifold}
Let $\Pi \in \pi^{\sharp}_1\big(\C(M, \M), \{0\}\big)$, and assume that $\bL^c(\Pi)>0$. There exists a nontrivial $V\in\V_n(M)$, such that
\begin{itemize}
\item[(i)] $V\in C(S)$ for some critical sequence $S$ of $\Pi$;
\item[(ii)] $V$ has $c$-bounded first variation;
\item[(iii)] $V$ is $c$-almost minimizing in small annuli. 
\end{itemize}
\end{theorem}
\begin{proof}
First we can pick a critical sequence $S$ of $\Pi$ which has been pulled-tight by Proposition \ref{P:tightening}, so that every ${V}\in C(S)$ has $c$-bounded first variation. Suppose for the sake of contradiction that for each ${V}\in C(S)$, there exists a $p\in M$, such that there are arbitrarily small annuli centered at $p$ on which ${V}$ is not $c$-almost minimizing. Then by Proposition \ref{P:def-equiv}, $V$ is also not $c$-almost minimizing with respect to the mass norm on these annuli (i.e. $\nu=\M$). 

Specifically, for any $\ti{r}>0$, there exists $r, s>0$ with $\ti{r}>r+2s>r-2s>0$, and $\ep>0$, such that for any $\de>0$, and $\Om\in\C(M)$ with $\mF(|\partial\Om|, V)<\ep$, then $\Om\notin \sA^c_n(A_{r-2s, r+2s}(p)\cap M; \ep, \de; \M)$. Now using the same argument as in \cite[4.10]{P81} by changing the mass functional $\M$ to the $\Ac$-functional, one can construct a new $1$-homotopic sequence $\ti{S}$ which is homotopic to $S$, and $\bL^c(\ti{S})<\bL^c(S)$; but this contradicts the criticality of $S$. 
\end{proof}

Now we formulate and solve a natural constrained minimization problem which will be used in the construction of $c$-replacements.
\begin{lemma}[A constrained minimization problem]
\label{L:minimisation}
Given $\ep, \de>0$, $U\subset M$ and any $\Om \in \sA^c_n(U;\ep,\de;\F)$, fix a compact subset $K\subset U$. Let $\C_\Om$ be the set of all $\La\in\C(M)$ such that there exists a sequence $\Om=\Om_0, \Om_1, \cdots, \Om_m=\La$ in $\C(M)$ satisfying:
 \begin{itemize}
\item[(a)] $\spt(\Om_i-\Om)\subset K$;
\item[(b)] $\F(\partial\Om_i - \partial\Om_{i+1})\leq \de$;
\item[(c)] $\Ac(\Om_i)\leq \Ac(\Om)+\de$, for $i=1, \cdots, m$.
\end{itemize}

Then there exists $\Omega^* \in \C(M)$ such that:
\begin{itemize}
\item[(i)] $\Om^* \in \C_\Om$, and \[ \Ac(\Om^*)=\inf\{\Ac(\La):\ \La\in\C_\Om\},\]
\item[(ii)] $\Om^*$ is locally $\Ac$-minimizing in $\interior(K)$,
\item[(iii)] $\Om^*\in \sA^c_n(U;\ep,\de;\F)$.
\end{itemize}
\end{lemma}
\begin{proof}
\textit{Proof of (i):} Take any minimizing sequence $\{\La_j\}\subset\C_{\Om}$, i.e.
\[ \lim_{j \to \infty} \Ac(\La_j) = \inf\{\Ac(\La):\ \La\in\C_{\Om}\}.\]
Notice that $\spt(\La_j-\Om) \subset K$ and $\Ac(\La_j) \leq \Ac(\Om) + \de$ for all $j$. By standard compactness \cite[Theorem 6.3]{Si83}, after passing to a subsequence, $\partial\La_j$ converges weakly to some $\partial\Om^*$ with $\Om^* \in\C(M)$ and $\spt(\Om^*-\Om) \subset K$.
We will show that $\Om^*$ is our desired minimizer. Since $\partial\La_j$ converges weakly to $\partial\Om^*$, we have that $\mH^n(\partial\Om^*)\leq \lim_{j\to\infty}\mH^n(\partial\La_j)$ and $\mH^{n+1}(\Om^*)=\lim_{j\to\infty}\mH^{n+1}(\La_j)$. Therefore,
\begin{equation}
\label{E:Om^*}
\Ac(\Om^*) \leq  \inf\{\Ac(\La):\ \La\in\C_{\Om}\}.
\end{equation}
It remains to show that $\Om^*\in \C_{\Om}$. For $j$ sufficiently large, we have $\F(\partial\La_j-\partial\Om^*)<\de$. Since $\La_j \in \C_{\Om}$, there exists a sequence $\Om=\Om_0, \Om_1, \cdots, \Om_m=\La_j$ in $\C(M)$ satisfying conditions (a-c) above. Consider now the sequence $\Om=\Om_0, \Om_1, \cdots, \Om_m=\La_j, \Om_{m+1}=\Om^*$ in $\C(M)$; it trivially satisfies conditions (a) and (b). Moreover, using (\ref{E:Om^*}), we also have
\[ \Ac(\Om^*)\leq \Ac(\La_j)\leq \Ac(\Om)+\de.\]
Therefore, $\Om^*\in \C_{\Om}$ and hence (i) has been proved. 

\textit{Proof of (ii):} For $p \in \interior(K)$, we claim that there exists a small $\sB_r(p) \subset \interior(K)$ such that
\begin{equation}
\label{E:Om^*2}
\Ac(\Om^*)\leq \Ac(\La),
\end{equation}
for any $\La\in\C(M)$ with $\spt(\La-\Om^*)\subset \sB_r(p)$. To establish (\ref{E:Om^*2}), first choose $r>0$ small so that $c\cdot \vol(\sB_r(p))<\de/4$ and $\M(\partial\Om^* \lc \sB_r(p)) < \de/4$ (this is possible since $\partial\Om^*$ is rectifiable). Suppose (\ref{E:Om^*2}) were false, then there exists $\Om'\in\C(M)$ with $\spt(\Om'-\Om^*)\subset \sB_r(p)$ such that $\Ac(\Om')<\Ac(\Om^*)$. We will show that $\Om'\in\C_{\Om}$, which contradicts that $\Om^*$ is a minimizer from part (i). 

To see that $\Om' \in \C_\Om$, take a sequence $\Om=\Om_0, \Om_1, \cdots, \Om_m=\Om^*$ in $\C(M)$ satisfying (a-c) above, and append $\Om_{m+1}=\Om'$ to the sequence. Since $\spt(\Om^*-\Om) \subset K$ and $\spt(\Om'-\Om^*) \subset K$, we have $\spt(\Om'-\Om)\subset K$. 
By the facts that $\spt(\Om'-\Om^*)\subset \sB_r(p)$ and $\Ac(\Om')<\Ac(\Om^*)$, we have
\begin{displaymath}
\begin{split}
\M(\partial \Om'\lc \sB_r(p)) & \leq \M(\partial\Om^*\lc \sB_r(p))+c\big[\mH^{n+1}(\Om^*\cap \sB_r(p))+\mH^{n+1}(\Om'\cap \sB_r(p))\big]\\
                                            & \leq \M(\partial\Om^*\lc \sB_r(p))+2 c\vol(\sB_r(p));
\end{split}
\end{displaymath}
hence
\[ \M(\partial \Om^{\pr}-\partial \Om^*)\leq \M(\partial \Om'\lc \sB_r(p)) +\M(\partial \Om^*\lc \sB_r(p)) <\de. \]
So $\F(\partial\Om'-\partial\Om^*) \leq \M(\partial \Om'-\partial \Om^*)<\de$. Finally note $\Ac(\Om') < \Ac(\Om^*) \leq \Ac(\Om)+ \de$. Therefore $\Om'\in\C_{\Om}$, and this proves part (ii). 

\textit{Proof of (iii):} Suppose that the claim is false. Then by Definition \ref{D:c-am-varifolds} there exists a sequence $\Om^*=\Om^*_0, \Om^*_1, \cdots, \Om^*_\ell $ in $\C(M)$ satisfying
\begin{itemize}
\item $\spt(\Om^*_i-\Om^*)\subset U$;
\item $\F(\partial\Om^*_i - \partial\Om^*_{i+1})\leq \de$;
\item $\Ac(\Om^*_i)\leq \Ac(\Om^*)+\de$, for $i=1, \cdots, \ell$,
\end{itemize}
but $\Ac(\Om^*_\ell) < \Ac(\Om^*)-\ep$. Since $\Om^* \in \C_\Om$ by part (i), there exists a sequence $\Om=\Om_0, \Om_1, \cdots, \Om_m=\Om^*$ satisfying conditions (a-c) above. Then the sequence $\Om=\Om_0,\Om_1,\cdots, \Om_m, \Om^*_1, \cdots, \Om^*_\ell$ in $\C(M)$ still satisfies those conditions (a-c), since $\Ac(\Om^*)\leq \Ac(\Om)$ implies that $\Ac(\Om^*_i)\leq \Ac(\Om)+\de$. Therefore $\Om\in \sA^c_n(U;\ep,\de;\F)$ implies that $\Ac(\Om^*_\ell)\geq \Ac(\Om)-\ep\geq \Ac(\Om^*)-\ep$, which is a contradiction. This proves part (iii). 
\end{proof}


\begin{proposition}[Existence and properties of replacements]
\label{P:good-replacement-property}
Let $V\in\V_n(M)$ be $c$-almost minimizing in an open set $U \subset M$ and $K \subset U$ be a compact subset, then there exists $V^{*}\in \V_n(M)$, called \emph{a $c$-replacement of $V$ in $K$} such that
\begin{enumerate}
\item[(i)] $V\lc (M\backslash K) =V^{*}\lc (M\backslash K)$;
\item[(ii)] $-c \vol(K)\leq \|V\|(M)-\|V^{*}\|(M) \leq c \vol(K)$;
\item[(iii)] $V^{*}$ is $c$-almost minimizing in $U$;
\item[(iv)] moreover, $V^{*} =\lim_{i \to \infty} |\partial\Om^*_i|$ as varifolds for some $\Om^*_i\in\C(M)$ such that $\Om^*_i\in \sA^c_n(U; \ep_i, \de_i; \F)$ with $\ep_i, \de_i \to 0$; furthermore $\Om^*_i$ locally minimizes $\Ac$ in $\interior(K)$;
\item[(v)] if $V$ has $c$-bounded first variation in $M$, then so does $V^*$.
\end{enumerate}
\end{proposition}

\begin{proof}
Let $V\in \V_n(M)$ be $c$-almost minimizing in $U$. By definition there exists a sequence $\Om_i\in \sA^c_n(U; \ep_i, \de_i; \F)$ with $\ep_i, \de_i \to 0$ such that $V$ is the varifold limit of $|\partial\Om_i|$. By Lemma \ref{L:minimisation} we can construct a $c$-minimizer $\Om_i^* \in \C_{\Om_i}$ for each $i$. Since $\M(\partial\Om_i^*)$ is uniformly bounded, by compactness there exists a subsequence $|\partial\Om_i^*|$ converging as varifolds to some $V^* \in \V_n(M)$. We claim that $V^*$ satisfies items (i)-(v) in Proposition \ref{P:good-replacement-property} and thus is our desired $c$-replacement. 
\begin{itemize}[leftmargin=0.8cm]
\item First, by part (i) of Lemma \ref{L:minimisation}, we have $\Om_i^* \in \C_{\Om_i}$ and thus $\spt(\Om_i^*- \Om_i) \subset K$. Hence the varifold limits satisfy $V^* \lc (M \backslash K) = V \lc (M \backslash K)$. 
\item Second, as $\Om_i \in \sA^c_n(U,\ep_i,\de_i; \F)$ and $\Om_i^* \in \C_{\Om_i}$, we have 
\[  \Ac(\Om_i)-\ep_i\leq \Ac(\Om_i^*)\leq \Ac(\Om_i);  \]
thus by (\ref{E: Ac}), 
\[ \M(\partial\Om_i)-c\mH^{n+1}(\Om_i)-\ep_i\leq \M(\partial\Om_i^*)-c\mH^{n+1}(\Om_i^*)\leq \M(\partial\Om_i)-c\mH^{n+1}(\Om_i). \] 
Note that $|\mH^{n+1}(\Om_i)-\mH^{n+1}(\Om_i^*)|\leq \vol(K)$; taking $i \to \infty$, we have $-c \vol(K)\leq \|V\|(M)-\|V^{*}\|(M) \leq c vol(K)$. 
\item Since each $\Om_i^* \in \sA^c_n(U; \ep_i, \de_i; \F)$ by Lemma \ref{L:minimisation}(iii), by definition $V^*$ is $c$-almost minimizing in $U$. 
\item (iv) follows from Lemma \ref{L:minimisation}(ii). 
\item Finally by (iii) and Lemma \ref{L:c-am implies c-bd-first-variation}, $V^*$ has $c$-bounded first variation in $U$. By (i) and a standard cutoff trick it is easy to show that $V^*$ has $c$-bounded first variation in $M$ whenever $V$ does. 
\end{itemize}
\end{proof}

\begin{lemma}[Regularity of $c$-replacement]
\label{L:reg-replacement}
Let $2 \leq n \leq 6$. Under the same hypotheses as Proposition \ref{P:good-replacement-property}, if $\Sigma=\spt \|V^*\| \cap \interior(K)$, then 
\begin{enumerate}
\item $\Sigma$ is a smooth, almost embedded, stable $c$-boundary;
\item the density of $V^*$ is $1$ along $\mR(\Si)$ and $2$ along $\mS(\Si)$;
\item the restriction of the $c$-replacement $V^* \lc \interior(K)=\Sigma$.
\end{enumerate}
\end{lemma}
\begin{proof}
By the regularity for local minimizers of the $\Ac$ functional (Theorem \ref{T:regularity of Ac minimizers}), we know that each $\partial\Om^*_i$ is a smooth, embedded, stable $c$-boundary in $\interior(K)$ by Proposition \ref{P:good-replacement-property}(iv). The lemma then follows from the compactness Theorem \ref{T:compactness}.
\end{proof}

Using Proposition \ref{P:good-replacement-property}, we can obtain the following preliminary lemma. It will be essential for proving that various blowups of the $c$-min-max varifold are planar, particularly in the proposition to follow, and in Section \ref{S:Regularity for c-min-max varifold}. 
Given $p\in M, r>0$, let $\bleta_{p,r}: \R^L\to\R^L$ be the dilation defined by $\bleta_{p, r}(x)=\frac{x-p}{r}$.
\begin{lemma}
\label{L:blowup is regular}
Let $2\leq n\leq 6$, and $V\in\V_n(M)$ be a $c$-almost minimizing varifold in $U$. Given a sequence $p_i\in U$ with $p_i\to p\in U$ and, a sequence $r_i>0$ with $r_i\to 0$, let $\overline{V}=\lim (\bleta_{p_i, r_i})_\# V$ be the varifold limit. Then $\overline{V}$ is an integer multiple of some complete embedded minimal hypersurface $\Si$ in $T_p M$, and moreover, $\Si$ is proper.
\end{lemma}
\begin{proof}
By Lemma \ref{L:c-am implies c-bd-first-variation}, $V$ has $c$-bounded first variation in $U$, so the blowup $\overline{V}$ is stationary in $T_pM$. We will show that $\overline{V}$ satisfies the good replacement property (Definition \ref{D:CD good replacement property}) in any open subset $W\subset T_p M$, which will imply that $\overline{V}$ is regular by Proposition \ref{P:CD regularity}. The properness follows from the Euclidean volume growth of $\overline{V}$, which is a direct corollary of the monotonicity formula.

Now fix a bounded open subset $W\subset T_p M$ and an arbitrary $x\in W$. Consider an arbitrary annulus $An=A_{s, t}(x)\subset W$ of outer radius $t\leq 1$. We will show that $\overline{V}$ has a replacement in $An$ in the sense of Definition \ref{D:CD good replacement}. Since $\bleta_{p_i, r_i}(M)\to T_p M$ locally uniformly, we can identify $\bleta_{p_i, r_i}(M)$ with $T_p M$ on compact subsets for $i$ large. Denote $An_i= \bleta_{p_i, r_i}^{-1}(An)$, then $\overline{An_i}\subset U$ for $i$ large.  For each $i$ large, we may apply Proposition \ref{P:good-replacement-property} to obtain a $c$-replacement $V^*_i$ of $V$ in $\overline{An_i}$. Denote $\overline{V}^*_i=(\bleta_{p_i, r_i})_\#V^*_i$, $\overline{V}_i=(\bleta_{p_i, r_i})_\# V$, then up to taking subsequences we have 
\[ \overline{V}'=\lim_{i\to\infty} \overline{V}^*_i, \text{ for some varifold $\overline{V}'$ in $T_pM$.} \] 
Moreover, we can deduce the following for $\overline{V}'$: 
\begin{itemize}[leftmargin=0.8cm]
\item Proposition \ref{P:good-replacement-property}(i) $\Longrightarrow$ $\overline{V}^*_i \lc (W\backslash\overline{An})= \overline{V}_i\lc (W\backslash\overline{An})$, hence \[\overline{V}' \lc (W\backslash\overline{An})= \overline{V}\lc (W\backslash\overline{An});\]

\item Proposition \ref{P:good-replacement-property}(ii) $\Longrightarrow$ $-c\vol(\overline{An_i})\cdot r_i^{-n}\leq \|\overline{V}^*_i\|(W)-\|\overline{V}_i\|(W)\leq c\vol(\overline{An_i})\cdot r_i^{-n}$.

Since the outer radius of $\overline{An}_i$ is at most $r_i$, for large $i$ there exists some $C_0>0$ depending only on $M$ such that $\vol(\overline{An}_i)\leq C_0 r_i^{n+1}$. This implies that \[\|\overline{V}'\|(W)=\|\overline{V}\|(W);\]

\item Proposition \ref{P:good-replacement-property}(iii) $\Longrightarrow$ $V_i^*$ has $c$-bounded first variation in $U$, hence $\overline{V}'$ is stationary in $W$;

\item By Lemma \ref{L:reg-replacement}, the restriction $V^*_i \lc An_i$ is a smooth, almost embedded, stable $c$-boundary $\Si_i^*$. Consider the rescalings: $\overline{\Si}_i=\bleta_{p_i, r_i}(\Si_i^*)\subset An$. By Proposition \ref{P:good-replacement-property}(ii) and the monotonicity formula \cite[40.2]{Si83}, $\overline{\Si}_i$ have uniformly bounded mass. This together with the compactness Theorem \ref{T:compactness}(iii) implies that $\overline{V}'\lc An$ is an embedded stable minimal hypersurface.
\end{itemize} 
Therefore, $\overline{V}'$ is a good replacement of $\overline{V}$ in $An$. By Proposition \ref{P:good-replacement-property}(iii), each $V^*_i$ is still $c$-almost minimizing in $U$. Hence for any other annulus $An'\subset W$ of outer radius $\leq 1$, we can repeat the above process and produce a good replacement $\overline{V}''$ of $\overline{V}'$ in $An'$. In fact, we may repeat this process any finite number of times. In particular, $\overline{V}$ satisfies the good replacement property (Definition \ref{D:CD good replacement property}) in $W$, and this completes the proof. 
\end{proof}

\begin{proposition}[Tangent cones are planes]
\label{P:tangent-cone}
Let $2 \leq n \leq 6$. Suppose $V \in \V_n(M)$ has $c$-bounded first variation in $M$ and is $c$-almost minimizing in small annuli. Then $V$ is integer rectifiable. Moreover, for any $C \in \VarTan(V,p)$ with $p \in \spt\|V\|$, 
\begin{equation}
\label{E:tangent cones are planes}
C= \Theta^n (\|V\|, p) |S| \text{ for some $n$-plane $S \subset T_p M$ where $\Theta^n(\|V\|, p)\in\N$}.
\end{equation}
\end{proposition}
\begin{proof}
Let $r_i\to 0$ be a sequence such that $C$ is the varifold limit:
\[ C=\lim_{i\to\infty} (\bleta_{p, r_i})_{\#} V. \]
First we know $C$ is stationary in $T_p M$. Since $V$ is $c$-almost minimizing in small annuli centered at $p$, by the same argument as in Lemma \ref{L:blowup is regular}, we can show that $C$ satisfies the good replacement property (Definition \ref{D:CD good replacement property}) in $T_p M$. (Note that Definition \ref{D:CD good replacement property} only requires the existence of good replacements in small annuli.) Therefore, by Proposition \ref{P:CD regularity}, $C$ is an integer multiple of some embedded minimal hypersurface of $T_p M$, and moreover, it is a cone by \cite[19.3]{Si83}. In particular $C$ is smooth and hence $\spt\|C\|$ must be a plane. This finishes the proof.
\end{proof}

\section{Regularity for $c$-min-max varifold}
\label{S:Regularity for c-min-max varifold}

In this section, we prove the regularity of our min-max varifolds. In particular we prove that every varifold which has $c$-bounded variation and is $c$-almost minimizing in small annuli is a smooth, closed, almost embeded, CMC hypersurface with multiplicity one.

\begin{theorem}[Main regularity]
\label{T:main-regularity}
Let $2 \leq n \leq 6$, and $(M^{n+1}, g)$ be an $(n+1)$-dimensional smooth, closed Riemannian manifold. 
Suppose $V \in \V_n(M)$ is a varifold which
\begin{enumerate}
\item has $c$-bounded first variation in $M$ and 
\item is $c$-almost minimizing in small annuli,
\end{enumerate}
then $V$ is induced by $\Sigma$,
where 
\begin{itemize}
\item[(i)] $\Sigma$ is a closed, almost embedded $c$-hypersurface (possibly disconnected);
\item[(ii)] the density of $V$ is exactly $1$ at the regular set $\mR(\Si)$ and $2$ at the touching set $\mS(\Si)$.
\end{itemize}
\end{theorem}

\begin{proof}
The conclusion is purely local, so we only need to prove the regularity of $V$ near an arbitrary point $p\in \spt\|V\|$. Fix a $p\in\spt\|V\|$, then there exists $0< r_0 < r_{am}(p)$ such that for any $r<r_0$, the mean curvature $H$ of $\partial B_r(p)\cap M$ in $M$ is greater than $c$. Here $r_{am}(p)$ is as in Definition \ref{D:am-annuli}.

In particular, if $r<r_0$ and $W\in\V_n(M)$ has $c$-bounded first variation in $B_r(p)\cap M$ and $W\neq 0$ in $B_r(p)$, then by the maximum principle (Proposition \ref{P:maximum principle})
\begin{equation}
\label{E:corollary-max-principle}
\emptyset \neq \spt \|W\| \cap \partial B_r(p) = \Clos\left( \spt \|W\| \setminus \Clos(B_r(p))\right) \cap \partial B_r(p).
\end{equation} 
Note that in the second equality we need a localized version of Proposition \ref{P:maximum principle} which holds true by the remark after \cite[Theorem 2]{White10}.

We will show that $V\lc B_{r_0}(p)$ is an almost embedded hypersurface of constant mean curvature $c$ with density equal to $2$ along its touching set. The argument consists of five steps:

\vspace{.3cm}
\noindent \textbf{Step 1:} \textit{Constructing successive $c$-replacements $V^*$ and $V^{**}$ on two overlapping concentric annuli.}

\noindent \textbf{Step 2:} \textit{Gluing the $c$-replacements smoothly (as immersed hypersurfaces) on the overlap.}

\noindent \textbf{Step 3:} \textit{Extending the $c$-replacements down to the point $p$ to get a $c$-`replacement' $\ti{V}$ on the punctured ball.}

\noindent \textbf{Step 4:} \textit{Showing that the singularity of $\ti{V}$ at $p$ is removable, so that $\ti{V}$ is regular.}

\noindent \textbf{Step 5:} \textit{$V$ coincides with the almost embedded hypersurface $\ti{V}$ on a small neighborhood of $p$.}
\vspace{.3cm}

We now proceed to the proof.

\subsection*{Step 1}
We first describe the construction of $c$-replacements on overlapping annuli; a key property will be that the replacements are also boundaries in the chosen annulus (see Claim 1). 

Fix $0<s<t<r_0$. By the choice of $r_0$, we can apply Proposition \ref{P:good-replacement-property} to $V$ to obtain a $c$-replacement $V^*$ in $K=\Clos(A_{s, t}(p)\cap M)$.  By (\ref{E:corollary-max-principle}) and Lemma \ref{L:reg-replacement}, the restriction
\[ \Si_1=V^*\lc (A_{s, t}(p)\cap M) \]
is a nontrivial, smooth, almost embedded, stable $c$-boundary with outer unit normal $\nu_1$. 

By Proposition \ref{P:smooth touching set}, the touching set $\mathcal{S}(\Sigma_1)$ is contained in a countable union of $(n-1)$-dimensional connected submanifolds $\bigcup S_1^{(k)}$. Since a countable union of sets of measure zero still has measure zero, it follows from Sard's theorem that we may choose $s_2 \in (s, t)$ such that $\partial (B_{s_2}(p)\cap M)$ intersects $\Si_1$ and all the $S_1^{(k)}$ transversally. 

Given any $s_1\in (0, s)$, by Proposition \ref{P:good-replacement-property}(iii), we can apply Proposition \ref{P:good-replacement-property} again to get a $c$-replacement $V^{**}$ of $V^*$ in $K=\Clos(A_{s_1, s_2}(p)\cap M)$. By (\ref{E:corollary-max-principle}) and Lemma \ref{L:reg-replacement} again, the restriction
\[ \Si_2=V^{**}\lc (A_{s_1, s_2}(p)\cap M) \]
is also a nontrivial, smooth, almost embedded, stable $c$-boundary with outer unit normal $\nu_2$. Note that by Proposition \ref{P:good-replacement-property}(v), both $V^*$ and $V^{**}$ have $c$-bounded first variation.

We can choose the second $c$-replacement $V^{**}$ so that it satisfies:

\vspace{.3cm}
\textit{Claim 1: there exists a set $\Om^{**}\in\C(M)$, such that 
\begin{itemize}
\item[a)] $\Si_1\cap A_{s_2, t}(p)$ and $\Si_2$ are the boundaries of $\Om^{**}$ in $A_{s_2, t}(p)$ and $A_{s_1, s_2}(p)$ respectively;
\item[b)] $\nu_1, \nu_2$ coincide with the outer unit normal of $\Om^{**}$ in $A_{s_2, t}(p)$ and $A_{s_1, s_2}(p)$ respectively;
\item[c)] if $\|V^{**}\|(\partial B_{s_2}(p))=0$, then $V^{**}$ is identical to $|\partial\Om^{**}|$ in $A_{s_1, t}(p)\cap M$.
\end{itemize}}

\textit{Proof of Claim 1:}  Fix $0<\tau<s_1$, then by Proposition \ref{P:good-replacement-property}(iv), $V^{*} =\lim_{i \to \infty} |\partial\Om^*_i|$ as varifolds for some $\Om^*_i\in\C(M)$, where $\Om^*_i\in \sA^c_n(A_{\tau, r_0}(p)\cap M; \ep_i, \de_i; \F)$ with $\ep_i, \de_i \to 0$. The regularity of the $\partial \Omega_i^*$ as in the proof of Lemma \ref{L:reg-replacement}, together with the compactness Theorem \ref{T:compactness} imply that $\partial\Om^*_i\lc A_{s, t}(p)$ converges locally smoothly to $\Si_1$ in $A_{s, t}(p)\cap M$. 

 Applying Lemma \ref{L:minimisation} for each $\Om^*_i$ in $K=\Clos(A_{s_1, s_2}(p)\cap M)$, we can construct new $c$-minimizers $\Om^{**}_i$, such that 
\begin{itemize}
\item $\spt(\Om^*_i-\Om^{**}_i)\subset K=\Clos(A_{s_1, s_2}(p)\cap M)$;
\item $\Om^{**}_i$ is locally $\Ac$-minimizing in $A_{s_1, s_2}(p)\cap M$;
\item $V^{**}=\lim |\partial\Om^{**}_i|$ as varifolds;
\item $\partial\Om^{**}_i\lc A_{s_1, s_2}(p)$ converges locally smoothly to $\Si_2$ (again as in the proof of Lemma \ref{L:reg-replacement}). 
\end{itemize} 
By the weak compactness \cite[Theorem 6.3]{Si83}, up to a subsequence, $\partial\Om^{**}_i$ converges weakly as currents to some $\partial\Om^{**}$ with $\Om^{**}\in\C(M)$. Claims (a) and (b) follow from the locally smooth convergence. The weak convergence implies that $\|\partial\Om^{**}\|(A_{s_1, t}(p))\leq \|V^{**}\|(A_{s_1, t}(p))$. If $\|V^{**}\|(\partial B_{s_2}(p))=0$, then together with the locally smooth convergence, we have $\|\partial\Om^{**}\|(A_{s_1, t}(p))= \|V^{**}\|(A_{s_1, t}(p))$;  moreover, $V^{**}\lc (A_{s_1, t}(p)\cap M)=|\partial \Om^{**}|\lc (A_{s_1, t}(p)\cap M)$ by \cite[2.1(18)(f)]{P81}. This confirms (c). \qed

\subsection*{Step 2}

We now show that $\Sigma_1$ and $\Sigma_2$ glue smoothly (as immersed hypersurfaces) across $\partial (B_{s_2}(p)\cap M)$. 
Indeed, define the intersection set
\begin{equation}
\Ga=\Si_1\cap \partial(B_{s_2}(p)\cap M),\quad \mS(\Ga)=\Ga\cap \mS(\Si_1).
\end{equation}
Then by transversality, $\Ga$ is an almost embedded hypersurface in $\partial(B_{s_2}(p)\cap M)$, and $\mS(\Ga)$ is its touching set.  Notice that
\begin{equation}
\label{E:openness of Gamma}
\text{$\mS(\Ga)$ is closed, and $\mR(\Ga)=\Ga\backslash\mS(\Ga)$ is open in $\Ga$.}
\end{equation}
It follows from the maximum principle that 
\[ \Clos(\Si_2)\cap \partial(B_{s_2}(p)\cap M)\subset\Ga. \] 
Indeed, (\ref{E:corollary-max-principle}) implies that any $y \in \Clos(\Si_2)\cap \partial(B_{s_2}(p)\cap M)$ is also a limit point of $\spt\|V^{**}\|$ from the outer side of $\partial B_{s_2}(p)$, on which $V^{**}$ coincides with $\Sigma_1$. In fact, with a little more work we have

\vspace{.3cm}
\textit{Claim 2: $\Clos(\Si_2)\cap \partial(B_{s_2}(p)\cap M)=\Ga$, and then $\Si_1$ glues together continuously with $\Si_2$.}

\textit{Proof of Claim 2:}  By Proposition \ref{P:good-replacement-property}(i), we have
\begin{equation}
\label{E:agreement-with-sigma1}
V^*=V^{**}=\text{$\Si_1$,}\quad \text{ in $A_{s_2, t}(p)\cap M$.}
\end{equation} 
Given any $x\in \Ga$, using (\ref{E:agreement-with-sigma1}), Proposition \ref{P:tangent-cone} and the fact that $\Si_1$ meets $\partial (B_{s_2}(p)\cap M)$ transversally, we have
\begin{equation}
\label{E:tangent cones for the second replacement}
\VarTan(V^{**}, x) = \{ \Theta^n(\|V^*\|, x) \; |T_x \Si_1| \}.
\end{equation}
This implies that $x$ is a limit point of $\spt\|V^{**}\|$ from inside of $\partial B_{s_2}(p)$, and thus completes the proof of the claim. \qed
\vspace{.3cm}

As a direct corollary of (\ref{E:tangent cones for the second replacement}), \cite[Theorem 3.2(2)]{Si83} and Claim 1(c), we have
\begin{equation}
\label{E:V** is boundary}
\text{$\|V^{**}\|(\partial B_{s_2}(p))=0$, and hence $V^{**}=|\partial\Om^{**}|$ in $A_{s_1, t}(p)\cap M$.}
\end{equation}

\vspace{.3cm}
Furthermore, we will show that $\Si_1$ glues with $\Si_2$ in $C^1$, i.e. the tangent spaces of $\Si_1$ and $\Si_2$ agree along $\Ga$, with matching normals. Take an arbitrary $q\in \Gamma$. We will need to divide to two sub-cases:

\vspace{.3cm}

\noindent\textbf{Sub-case (A)}: $q$ is a regular point of $\Sigma_1$, i.e. $q\in \mR(\Ga)$.

\vspace{.3cm}

First we have the following.

\vspace{.3cm}
\textit{Claim 3(A): Fix $x\in\mR(\Ga)$, for any sequence of $x_i\to x$ with $x_i\in\mR(\Ga)$ and $r_i\to 0$, we have}
\[ \lim_{i\to\infty} (\bleta_{x_i, r_i})_{\#} V^{**}=  T_x\Si_1 \quad \text{ as varifolds}. \]

\textit{Proof of Claim 3(A):} By the weak compactness of Radon measures, after passing to a subsequence,
\begin{equation}
\label{E: limit varifold of V** at xi}
\lim_{i \to \infty} (\bleta_{x_i,r_i})_\sharp V^{**}=C \in\V_n(T_x M).
\end{equation}
By Lemma \ref{L:blowup is regular}, $C$ is a regular, proper, complete minimal hypersurface in $T_x M$. By (\ref{E:agreement-with-sigma1}), $C$ coincides with $T_x\Si_1$ on a half space of $T_x M$. The classical maximum principle implies that $C\supset T_x \Si_1$. It then follows from the half space theorem for minimal hypersurfaces \cite[Theorem 3]{HM90} that there are no other connected components of $C$. Thus $C=T_x\Si_1$ and the proof is complete. \qed

\vspace{.3cm}

Since $\{(\bleta_{x_i,r_i})_\sharp V^{**}: i\in\N\}$ have uniformly bounded first variation, a standard argument using the monotonicity formula implies that
\begin{equation}
\label{E:Hausdorff-convergence}
\spt \|(\bleta_{x_i,r_i})_\sharp V^{**}\| \to T_x \Sigma_1 \text{ in the Hausdorff topology}. 
\end{equation}
To show that $\Si_1$ and $\Si_2$ glue together along $\Ga$ in $C^1$ near $q$. We need to show that:

\vspace{.3cm}

\textit{Claim 4(A): For each $x \in \mR(\Ga)$, we have}
\[ \lim_{z \to x, z \in \Si_2} \nu_2(z) = \nu_1(x).\]
\textit{Moreover, the convergence is uniform in $x$ on compact subsets of $\Gamma$ near $q$.}

\textit{Proof of Claim 4(A):} The uniformity follows from the fact that $\nu_1$ is continuous on $\Gamma$, so we only need to establish the convergence to $\nu_1$.  

So consider a sequence $z_i\in \Sigma_2$ converging to some $x\in \mathcal{R}(\Ga)$. Since $x$ is a regular point of $\Si_1$, by Claim 3(A) and the upper semi-continuity of density function for varifolds with bounded first variation \cite[17.8]{Si83}, we know that $z_i$ is also a regular point of $\Si_2$ for $i$ large enough.

Take $x_i \in \Ga$ to be the nearest point projection (in $\R^L$) of $z_i$ to $\Ga$ and $r_i=|z_i-x_i|$. Note that $x_i \to x \in \mathcal{R}(\Ga)$ and $r_i \to 0$, so we are in the situation of Claim 3(A). Note that $\Si_2 \cap B_{r_i/2}(z_i)$ is an embedded, stable $c$-hypersurface in $M$, so by Theorem \ref{T:compactness} a subsequence of the blow-ups $\bleta_{x_i,r_i} (\Si_2 \cap B_{r_i/2}(z_i))$ converges smoothly to a smooth, embedded, stable, minimal hypersurface $\Si_\infty$ contained in a half-space of $T_x M$. 

On the other hand, Claim 3(A) and (\ref{E:Hausdorff-convergence}) imply that $\bleta_{x_i,r_i} (\Si_2 \cap B_{r_i/2}(z_i))$ converges in the Hausdorff topology to a domain in $T_x \Si_1$. Therefore, we have $\Si_\infty \subset T_x \Si_1$. The smooth convergence then implies that $\nu_2(z_i)$ converges to one of the unit normals $\pm\nu_1(x)$ of $T_x \Sigma_1$. By Claim 1 and (\ref{E:V** is boundary}), we know that $V^{**}=|\partial \Om^{**}|$ in $A_{s_1, t}(p)\cap M$ and $\|\partial \Om^{**}\|(\partial B_{s_2}(p))=0$, therefore the limit of the $\nu_2(z_i)$ must be $\nu_1(x)$, so \emph{Claim 4(A)} is proved. \qed

\vspace{.3cm}

Thus we have proven that near any regular point $q\in \mR(\Ga)$, $\Sigma_1$ and $\Sigma_2$ glue together along $\Gamma$ as a $C^1$ hypersurface with matching outer unit normals $\nu_1,\nu_2$. The higher regularity follows from a standard elliptic PDE argument. More precisely, $\Si_1$ and $\Si_2$ can be written as graphs of some functions $u_1, u_2$ over $T_q\Si_1$ respectively. Since the mean curvatures of both $\Si_1$ and $\Si_2$ are identical to $c>0$ with respect to some unit normals pointing to the same side of $T_q\Si_1$, they satisfy the same mean curvature type elliptic PDE with in-homogenous term equal to $c$. The higher regularity follows from the elliptic regularity of this PDE. 
This finishes {\bf Sub-case (A)}.

\vspace{.3cm}

At this point we have proven that $\Si_2$ glues smoothly with $\Si_1\cap A_{s_2, t}(p)$ along $\mR(\Ga)$. Moreover, by the unique continuation for elliptic PDE, we know that $\Si_2$ is identical to $\Si_1$ 
in a neighborhood of $\mR(\Ga)$ in $A_{s_2, t}(p)\cap M$.
We will need to show that the smooth gluing extends to the touching set $\mS(\Ga)$.

\vspace{.3cm}

\noindent\textbf{Sub-case (B)}: $q$ is a touching point of $\Sigma_1$, i.e. $q\in \mS(\Ga)\subset \mS(\Si_1)$.

\vspace{.3cm}

By Lemma \ref{L:reg-replacement}, in some small neighborhood $U\subset M$ of $q$, $\Si_1\cap U$ is the union of two connected, embedded $c$-hypersurfaces $\Si_{1,1}\cup \Si_{1,2}$ with unit normals $\nu_{1,1}$ and $\nu_{1,2}$, such that $\Si_{1,2}$ lies on one side of $\Si_{1,1}$ and they touch tangentially at $\mS(\Si_1)\cap U=\Si_{1,1}\cap\Si_{1,2}$. By Lemma \ref{L:classical MP}, $\nu_{1,1}=-\nu_{1,2}$ along the touching set $\mS(\Si_1)\cap U$. Denote $\Ga\cap \Si_{1,1}=\Ga_1$ and $\Ga\cap \Si_{1,2}=\Ga_2$, then as embedded submanifolds of $\partial (B_{s_2}(p)\cap U)$, $\Ga_2$ lies on one-side of $\Ga_1$ and they touch tangentially along $\mS(\Ga)\cap U$.

\vspace{.3cm}
\textit{Claim 3(B): Fix $x\in\mS(\Ga)$ and denote $P_x=T_x \Sigma_{1,1}=T_x\Sigma_{1,2}$. For any sequence of $x_i\to x$ with $x_i\in \Ga$ and $r_i\to 0$, up to a subsequence we have}
\[ \lim_{i \to \infty} (\bleta_{x_i,r_i})_\sharp V^{**} =\left\{ \begin{array}{cl}
P_x+\btau_v P_x & \text{ for Type I convergence} \\
2P_x & \text{ for Type II convergence} \end{array}, \right. \]
where $\btau_w$ denotes translation by a vector $w$, and $v\in (P_x)^\perp$ is a vector in $T_xM$ orthogonal to $P_x$ ($v$ may be $\infty$, in which case $\btau_v P$ is understood to be the empty set). The two convergence scenarios are:
\begin{itemize}
\item Type I: $\liminf_{i \to \infty} \dist_{\R^L}(x_i, \mS(\Ga))/r_i=\infty$,
\item Type II: $\liminf_{i \to \infty} \dist_{\R^L}(x_i, \mS(\Ga))/r_i < \infty$.
\end{itemize}

\textit{Proof of Claim 3(B):} 
First we determine the blowup limit $C' = \lim_{i\rightarrow \infty} (\bleta_{x_i,r_i})(\Si_1)$. 

In Type I convergence, for any $R>0$, $\Gamma\cap B_{r_i R}(x_i) \subset \mR(\Ga)$ for $i$ large enough. Up to a subsequence, we can assume that all $x_i$ belong to $\Ga_1$, then $(\bleta_{x_i,r_i})(\Si_{1,1})$ converges locally smoothly to $P_x$. Let $x_i^\pr$ be the nearest point projection of $x_i$ to $\Si_{1, 2}$, and let $v=\lim_{i \to \infty}\frac{x_i^\pr-x_i}{r_i}$ (up to taking a subsequence), which maybe $\infty$. If $v$ is finite, then $(\bleta_{x_i,r_i})(\Si_{1,2})$ converges locally smoothly to $P_x+v$; if $v$ is infinite, then $(\bleta_{x_i,r_i})(\Si_{1,2})$ disappears in the limit. So in this case $C' = P_x + \btau_v P_x$. In the Type II scenario, the touching set does not disappear in the limit and we have $C'=\lim (\bleta_{x_i,r_i})(\Si_1)=2P_x$. 

Now let $C=\lim_{i \to \infty} (\bleta_{x_i,r_i})_\sharp V^{**}$ be the varifold limit as in (\ref{E: limit varifold of V** at xi}). By Lemma \ref{L:blowup is regular}, $C$ is a regular, proper, complete minimal hypersurface in $T_xM$. Again $C$ must coincide with $C'=\lim_{i\rightarrow \infty} (\bleta_{x_i,r_i})(\Si_1)$ on a halfspace of $T_x M$. Since $C'$ consists of one or two parallel planes, the classical maximum principle implies that $C$ contains these planes, and again the halfspace theorem \cite[Theorem 3]{HM90} rules out any other components of $C$. Thus $C$ is identical to $C'$, which completes the proof.
\qed

\vspace{.3cm}
By Claim 3(B) and the same argument in Claim 4(A), we know that 
\begin{equation}
\label{E:convergence of tangent planes}
\lim_{z \to x, z \in \Si_2} [T_z\Si_2] = [T_x\Si_1],
\end{equation}
where $[T_z\Si_2]$ and $[T_x\Si_1]$ denote the un-oriented tangent planes of $\Si_2$ and $\Si_1$ (without counting multiplicity) respectively. Moreover, the convergence is uniform in $x$ on compact subsets of $\mS(\Ga)$ near $q$. Therefore using the regularity of $\Si_2$ in Lemma \ref{L:reg-replacement}, near $q$ the hypersurface $\Si_2$ can be written as a set of graphs $\{\Si_{2, i}: i=1\cdots l\}$ over the half space $[T_q(\Si_1\cap B_{s_2}(p))]$.

Indeed, take $\rho$ small so that $B_\rho(q) \cap M$ may be identified with the tangent space $T_q M$, and for $z\in \Sigma_2 \cap B_\rho(q)$ let $C_\epsilon(z)$ denote (the image of) the cylinder in $T_q M$ with axis perpendicular to $T_q\Sigma_1$ and radius $\epsilon$. For small enough $\rho,\epsilon$, by almost-embeddedness and the uniform convergence of tangent planes, $\Sigma_2 \cap C_\epsilon(z)\cap B_\rho(q)$ decomposes as a finite number of ordered graphs over $T_q\Sigma_1 \cap C_\epsilon(z)$, which have uniformly bounded gradient $\delta \ll 1$. The uniform gradient bound, together with unique continuation for CMC hypersurfaces imply that this graphical decomposition may be extended all the way to the boundary $\Gamma$, preserving the ordering and the gradient bound.

Now since $\Si_2$ glues smoothly with $\Si_1$ along $\mR(\Ga)$, and since $\mR(\Ga)$ is an open and dense subset of $\Ga$, we know that the set $\{\Si_{2, i}: i=1\cdots l\}$ consists of exactly two elements: one of them, denoted by $\Si_{2,1}$, glues smoothly with $\Si_{1,1}$ along $\Ga_1\backslash \mS(\Ga)$; the other one, denoted by $\Si_{2, 2}$, glues smoothly with $\Si_{1,2}$ along $\Ga_2\backslash \mS(\Ga)$. Now (\ref{E:convergence of tangent planes}) implies that the pairs $(\Si_{1,1}, \Si_{2,1})$ and $(\Si_{1,2}, \Si_{2,2})$ glue together in $C^1$ near $q$ respectively. Again higher regularity follows from the elliptic PDE argument as in Sub-case (A). This finishes {\bf Sub-case (B)}.

\subsection*{Step 3}
We now wish to extend the replacements, via unique continuation, all the way to the center $p$. 

Henceforth we denote $V^{**}$ by $V^{**}_{s_1}$ and $\Si_2$ by $\Si_{s_1}$ to indicate the dependence on $s_1$. By varying $s_1\in (0, s)$, we obtain a family of nontrivial, smooth, almost embedded, stable $c$-boundaries $\{\Si_{s_1}\subset A_{s_1, s_2}(p)\cap M\}$. Since unique continuation holds for immersed CMC hypersurfaces, by Step 2 we have $\Si_{s_1}=\Si_1$ in $A_{s, s_2}(p)$, and moreover, for any $s_1'<s_1<s$, we have $\Si_{s_1'}=\Si_{s_1}$ in $A_{s_1, s_2}(p)$. Hence 
\[ \Si:=\bigcup_{0<s_1<s} \Si_{s_1}\]
is a nontrivial, smooth, almost embedded, stable $c$-hypersurface in $(B_{s_2}(p)\backslash\{p\})\cap M$. 
Also
\[V^{**}_{s_1}=\Si, \text{ in } A_{s_1, s_2}(p), V^{**}_{s_1}=V^* \text{ in } A_{s, s_2}(p),\]
\[ \text{and for any } s_1'<s_1<s, V^{**}_{s_1'}=V^{**}_{s_1} \text{ in } A_{s_1, s_2}(p). \]
By Proposition \ref{P:good-replacement-property}, $V^{**}_{s_1}$ has $c$-bounded first variation and uniformly bounded mass for all $0<s_1<s$, so the monotonicity formula \cite[40.2]{Si83} implies that $\|V^{**}_{s_1}\|(B_r(p))\leq C r^n$ for some uniform $C>0$. Therefore as $s_1\to 0$, the family $V^{**}_{s_1}$ will converge to a varifold $\ti{V}\in\V_n(M)$, i.e.
\[\ti{V}=\lim_{s_1\to 0} V^{**}_{s_1},\quad \text{such that}\]
\begin{equation}
\label{E:tilde V}
\ti{V} =\left\{ \begin{array}{cl}
\Si & \text{in } (B_{s_2}(p)\backslash\{p\})\cap M \\
V^* & \text{in } M\backslash B_s(p) \end{array}, \text{ and } \|\ti{V}\|(\{p\})=0.\right. 
\end{equation}

Since $p\in\spt\|V^{**}_{s_1}\|$, by the upper semi-continuity of density function for varifolds with bounded first variation \cite[17.8]{Si83}, we know that $p\in\spt\|\ti{V}\|$.

\subsection*{Step 4}

We now determine the regularity of $\ti{V}$ at $p$.

First, since $\ti{V}$ is the varifold limit of a sequence $V^{**}_{s_1}$ which all have $c$-bounded first variation, we know that $\ti{V}$ also has $c$-bounded first variation. Second, $\ti{V}$, when restricted to any small annulus $A_{r, 2r}(p)\cap M$, already coincides with a smooth, almost embedded, stable $c$-boundary $\Si$.  Using these two ingredients, we can use a blowup argument as in the proofs of Lemma \ref{L:blowup is regular} and Proposition \ref{P:tangent-cone} (without the need for replacements) to show that every tangent varifold of $\ti{V}$ at $p$ is an integer multiple of some $n$-plane, i.e. for any $C \in \VarTan(V,p)$, 
\[ C= \Theta^n (\|\ti{V}\|, p) |S| \text{ for some $n$-plane $S \subset T_p M$ where $\Theta^n(\|\ti{V}\|, p)\in\N$}. \]

Now the removability of the singularity of $\ti{V}$ at $p$ (as an almost embedded hypersurface) follows similarly to \cite[Theorem 7.12]{P81}. We include the details for completeness. We can assume that
\[ \Theta^n (\|\ti{V}\|,p)=m \]
for some $m \in \mathbb{N}$. Since $\Sigma$ is stable in a punctured ball of $p$, by Theorem \ref{T:compactness}, for any sequence $r_i \to 0$, 
\[ \bleta_{p,r_i}(\Sigma) \to m\cdot S \]
locally smoothly in $\mathbb{R}^L \setminus \{0\}$ for some $n$-plane $S \subset T_p M$. However, $S$ may depend on the sequence $r_i$. By the convergence and the regularity of $\Si$, there exists $\si_0>0$ small enough, such that for any $0<\si\leq \si_0$, $\Si$ has an $m$-sheeted, ordered (in the sense of Definition \ref{def:comparisonsheets}), graphical decomposition in $A_{\si/2, \si}(p)$:
\begin{equation}
\label{E: annuli-decomposition}
\ti{V}\lc A_{\si/2, \si}(p)=\sum_{i=1}^{m}|\Si_i(\si)|.
\end{equation}
Here $\Si_i(\si)$ is a graph over $A_{\si/2, \si}(p)\cap S$ for some $n$-plane $S\subset T_p M$.

Since (\ref{E: annuli-decomposition}) holds for all $\si$, by continuity of $\Si$ we can continue each $\Si_i(\si_0)$ to $(B_{\si_0}(p)\setminus\{p\})\cap M$, and we denote the continuation by $\Si_i$. Since each piece $\Si_i$ has constant mean curvature $c$, by a standard extension argument (c.f. the proof in \cite[Theorem 4.1]{HL75}), each $\Si_i$ can be extended as a varifold with $c$-bounded first variation in $B_{\si_0}(p)\cap M$. Given $C_i\in \VarTan(\Si_i, p)$, to see that $C_i$ has multiplicity one, first notice that
\begin{equation}
\label{E:density=1}
\Theta^n(\|C_i\|,p) \geq 1, 
\end{equation}
since each $\Si_i$ is $c$-stable and thus its re-scalings converge with multiplicity to a smooth, embedded, stable, minimal hypersurface by Theorem \ref{T:compactness}. If equality does not hold for some $i$ in (\ref{E:density=1}), this will derive a contradiction since 
\[ \tilde{V}\lc B_{\si_0}(p)= \sum_{i=1}^m |\Si_i|.\]
Therefore, each $\Si_i$ has $c$-bounded first variation in $B_{\si_0}(p) \cap M$ and $\Theta^n(\||\Sigma_i|\|,p)=1$; by the Allard regularity theorem  \cite[Theorem 24.2]{Si83} and elliptic regularity, $\Si_i$ extends as a smooth, embedded $c$-hypersurface across $p$. Finally, by the maximum principle (Lemma \ref{L:classical MP}), we have $m=1$ or $m=2$ and this shows that $\ti{V}$ extends as an almost embedded $c$-hypersurface across $p$.

\subsection*{Step 5}

Finally, to complete the proof we show that $V$ coincides with $\ti{V}$ on a small ball about $p$. 

We will need the following simple corollary of the first variation formula.

\begin{lemma}
\label{L: tangent set of V}
For small enough $r$ the set
\[ \text{Tr}^V_p=\left\{ y \in \spt\|V\|\cap (B_r(p)\backslash\{p\}) : \begin{array}{l}
\VarTan(V, y) \text{ consists of an integer multiple of an }\\
\text{$n$-plane transverse to } \partial (B_{\dist(y, p)}(p)\cap M)
\end{array} \right\}\]
is a dense subset of $\spt \|V\|\cap B_r(p)$.
\end{lemma}

\begin{proof}
 Assume to the contrary that there exists $y\in \spt \|V\|\cap (B_r(p)\backslash\{p\})$, and some some neighborhood $U$ of $p$ in $B_r(p)\cap M$ such that the tangent plane $S_z$ of $V$ at each point $z\in U\cap \spt\|V\|$ is tangential to $\partial (B_{\dist(z, p)}(p)\cap M)$. Since $M$ is embedded in $\R^L$, the family $\{\partial (B_{\rho}(p)\cap M): 0<\rho<r\}$ forms a smooth foliation of $(B_{r}(p)\cap M)\backslash\{p\}$ for small enough $r$. Let $\nu$ be the outer unit normal of this smooth foliation. By our choice of $r_0$, the mean curvature on the foliation is  
\[ div_{T_y (\partial (B_{\dist(y, p)}(p)\cap M))} \nu >c. \] 
Consider the vector field $X=\varphi\cdot \nu$, where $\varphi\geq 0$ is a smooth cutoff function supported in $U$. Plugging $X$ into the first variation formula, we get
\begin{displaymath}
\begin{split}
\de V(X) & =\int div_{S_z} (\varphi \nu) d\|V\|(z)=\int \varphi\cdot div_{S_z} (\nu) d\|V\|(z)\\
              & >c \cdot \int \varphi d\|V\|(z) = c \int |X| d\|V\|(z).
\end{split}
\end{displaymath}
In the second equality we have used the fact that $S_z$ is tangential to the foliation and hence perpendicular to $\nu$. This contradicts that $V$ has $c$-bounded first variation and thus finishes the proof.
\end{proof}


\textit{Claim 5: For small enough $r$, $\spt\|V\|=\Si$ in the punctured ball $(B_r(p)\backslash\{p\})\cap M$.}
 
\vspace{.3cm}
\textit{Proof of Claim 5:}  We first prove that $\text{Tr}^V_p\subset \Si$, which combined with Lemma \ref{L: tangent set of V} will imply that $\spt\|V\|\cap (B_r(p)\backslash\{p\})\subset \Si.$ Fix $y \in \text{Tr}^V_p \cap (B_r(p) \backslash \{p\})$, and let $\rho=\dist(y, p)$. Consider $V^{**}_\rho$. By transversality we have $ y  \in \Clos(\spt\|V\|\cap B_\rho(p))$. On the other hand, $V^{**}_{\rho}=V^*=V$ inside $B_\rho(p)$, so by (\ref{E:corollary-max-principle}) we have
\begin{eqnarray}\nonumber \Clos(\spt\|V\|\cap B_\rho(p)) \cap \partial B_\rho(p)&=&\Clos(\spt\|V^{**}_\rho\|\cap B_\rho(p)) \cap \partial B_\rho(p)\\&\subset&\nonumber \Clos\left( \spt \|V^{**}_\rho\| \setminus \Clos(B_\rho(p))\right) \cap \partial B_\rho(p).\end{eqnarray}
Since $\spt \|V^{**}_{\rho}\|= \Si$ on $A_{\rho, s_2}(p)$, we therefore have $y \in \Si$. 

Next we show the reverse inclusion $\Si\subset \spt\|V\|$. Since $\Si$ extends across $p$ as an almost embedded hypersurface, we know that $T_y\Si$ is transverse to $\partial (B_{\dist(y, p)}(p)\cap M)$ for all $y\in\Si\cap B_r(p)$ for small enough $r$. 
Let $\rho$ and $V^{**}_\rho$ be as above, then $y\in \spt\|V^{**}_\rho\|$. By Proposition \ref{P:tangent-cone}, $\VarTan(V^{**}_\rho, y)=\{ \Theta^n(\|V^{**}_\rho\|, y) |T_y \Si| \}$. By the transversality, we then have $y\in \Clos(\spt\|V^{**}_\rho\|\cap B_\rho(p))$, so since $V^{**}_\rho=V$ inside $B_\rho(p)$ we conclude that $y\in\Clos(\spt\|V\|\cap B_\rho(p))\subset \spt\|V\|$ as desired. \qed

\vspace{.3cm}

Note that we do not have the Constancy Theorem (c.f. \cite[41.1]{Si83}) for varifolds with bounded first variation. In order to show that $V$ coincides with $\Si$ near $p$, our strategy is to show that $V=\ti{V}$ as varifolds in a neighborhood of $p$. By the transversality argument as above, we can first show that the densities of $V$ and $\ti{V}$ are identical along $\Si\cap (B_r(p)\backslash\{p\}$, where $r$ is chosen as in Claim 5. 

\vspace{.3cm}
\textit{Claim 6: $\Theta^n(\|V\|, \cdot)=\Theta^n(\|\ti{V}\|, \cdot)$ on $\Si\cap B_r(p)\backslash\{p\}$.} 

\vspace{.3cm}
\textit{Proof of Claim 6:} 
Let $y\in\Si$ and $\rho=\dist(y,p) <r$ be as above. Then since $V^{**}_\rho=V$ inside $B_\rho(p)$, by transversality and Proposition \ref{P:tangent-cone} we have $\VarTan(V, y)=\VarTan(V^{**}_\rho, y)$. But $V^{**}_\rho = \tilde{V}$ on $A_{\rho,s_2}(p)$, so we must have $\VarTan(V^{**}_\rho, y)=\{ \Theta^n(\|\ti{V}\|, y)|T_y \Si| \}$. Thus $\Theta^n(\|V\|, y)=\Theta^n(\|\ti{V}\|, y)$.  \qed

Combining Claims 5 and 6 yields that $V=\tilde{V}$ on $B_r(p)\cap M$. This finishes the proof of Step 5, and hence also completes the proof of the main Theorem \ref{T:main-regularity}.\end{proof}

\appendix

\section{An interpolation lemma}
\label{A:An interpolation lemma}

The first lemma below was essentially due to Pitts \cite[Lemma 3.8]{P81}, but the modification to find the interpolation sequence using boundaries of Caccioppoli sets was completed by the first author \cite[Proposition 5.3]{Zhou15b}. 
\begin{lemma}
\label{L:interpolation1}
Suppose $L>0$, $\eta>0$, $W$ is a compact subset of $U$, and $\Om\in\C(M)$. Then there exists $\de=\de(L, \eta, U, W, \Om)>0$, such that for any $\Om_1, \Om_2\in \C(M)$ satisfying
\begin{itemize}
\item[$(a)$] $\spt(\Om_i-\Om)\subset W$, $i=1, 2$,
\item[$(b)$] $\M(\partial \Om_i)\leq L$, $i=1, 2$,
\item[$(c)$] $\F(\partial\Om_1-\partial\Om_2)\leq \de$,
\end{itemize}
there exist a sequence $\Om_1=\La_0, \La_1, \cdots, \La_m=\Om_2\in \C(M)$ such that for each $j=0, \cdots, m-1$,
\begin{itemize}
\item[(i)] $\spt(\La_j-\Om)\subset U$,
\item[(ii)] $\Ac(\La_j)\leq \max\{\Ac(\Om_1), \Ac(\Om_2)\}+\eta$,
\item[(iii)] $\M(\partial\La_j-\partial\La_{j+1})\leq \eta$.
\item[(iv)]
\[ \M(\partial\La_j)\leq \max\{\M(\partial\Om_1), \M(\partial\Om_2)\}+\frac{\eta}{2}.\]
\item[(v)] 
\[ \M(\La_j-\Om_i)\leq \frac{\eta}{2c}, \text{ for } i =1, 2. \]
\end{itemize}
\end{lemma}

\begin{proof}
Note that by a covering argument, one only needs to prove the case when $\partial\Om_2$ is fixed, and in this case $\de=\de(L, \eta, U, W, \Om, \Om_2)$. 

Under our assumptions on $\Om_1, \Om_2$, there are two issues for us that differ from the conclusions of Pitts \cite[Lemma 3.8]{P81}: 
\begin{enumerate}
\item we require the interpolating sequence $\{\partial\La_j\}$ to consist of boundaries of Caccioppoli sets, while in \cite{P81} the interpolating sequence consists of integral currents $\{T_j\in \Z_n(M)\}$ (using notations therein);
\item for point (ii), we require that $\Ac(\La_j)$ does not increase much from $\Ac(\Om_1), \Ac(\Om_2)$, while in \cite{P81} it was only proven that $\M(T_j)\leq \max\{\M(\partial\Om_1), \M(\partial\Om_2)\}+\eta$. 
\end{enumerate}

These two minor points can be easily deduced from the mentioned work of Zhou \cite{Zhou15b}, and now we point out the necessary details. In fact, with regards to point (1), the proof of \cite[Proposition 5.3]{Zhou15b} already proceeds using boundaries, and we will first justify properties $(i, iii, iv)$. 

Specifically, given $L, \eta, \Om, \Om_2$ satisfying the assumptions in the Lemma, \cite[Proposition 5.3]{Zhou15b} (applied to the case when $m=1, l=0$ therein) gives the desired $\de>0$, such that if $\Om_1$ satisfies the assumptions (a-c), then there exists a sequence $\Om_1=\La_0, \La_1, \cdots, \La_m=\Om_2\in \C(M)$ such that for each $j=0, \cdots, m-1$, properties (iii) and (iv) are satisfied.

Although in \cite[Proposition 5.3]{Zhou15b} $\Om_i$ were not assumed to satisfy the additional condition $\spt(\Om_i-\Om)\subset W$. It can be seen that the construction in \cite[Proposition 5.3]{Zhou15b} indeed satisfies the required property (i) under this additional condition. 

Note that properties (iv, v) immediately imply (ii), so it remains only to show (v). Indeed, from the construction in \cite[Proposition 5.3]{Zhou15b}, one knows that the symmetric difference $\La_j\lap \Om_i$ is a subset of the union of the symmetric difference $\Om_1\lap \Om_2$ together with finitely many (depending only on $L$ and $\eta$) balls of arbitrarily small radii in $U$. Since $\vol(\Om_1\lap\Om_2)=\M(\Om_1-\Om_2)=\F(\partial\Om_1, \partial\Om_2)$ by the isoperimetric lemma \cite[Lemma 7.3]{Zhou15b}, one thus obtains (v) by taking $\de$ small enough.
\end{proof}

\section{Interpolation process}
\label{A:proof of claim 2}

\begin{proof}[Proof of Claim 2 in Proposition \ref{P:tightening}]
Here we describe the construction of $\{\phi_i\}$ by interpolating $\{\phi^1_i\}$. 

Fix $i\in\N$ and consider a $1$-cell $\al\in I(1, k_i)$. We only need to show how to interpolate $\phi^1_i$ when restricted to $\al_0$. For notational simplicity we write $\al=[0, 1]$. For $x\in \alpha$ let $\ti{X}(x)$ be the linear interpolation between $\ti{X}_i(0)=\ti{X}(|\partial\phi^*_i(0)|)$ and $\ti{X}_i(1)=\ti{X}(|\partial\phi^*_i(1)|)$. The continuity of the map $V\to \ti{X}(V)$ implies that $\|\ti{X}_i(x)-\ti{X}_i(0)\|_{C^1(M)}\to 0$ uniformly as $i\to\infty$. Define $\bar{Q}_i(x)$ to be the push-forward of $\phi^*_i(0)$ by the flow of $\ti{X}_i(x)$ up to time $1$; this gives a map $\bar{Q}_i: \alpha \rightarrow \mathcal{C}(M)$. Note that $\partial\bar{Q}_i: \al\to \Z_n(M)$ is continuous under the $\mF$-metric.

Since $\bar{Q}_i(x)$ and $\phi^1_i(0)$ are the push-forwards of the same initial set $\phi^*_i(0)$ under the flows of $\ti{X}_i(x)$ and $\ti{X}_i(0)$ respectively, we have
\begin{equation}
\label{E:mF fineness control1}
\left. \begin{array}{cl}
\mF(\partial \bar{Q}_i(x), \partial \phi^1_i(0))\to 0 \\
\M(\bar{Q}_i(x)- \phi^1_i(0))\to 0
\end{array} \right.,  \text{ uniformly in $x,\alpha$ as $i\to\infty$}.
\end{equation}
As $\bar{Q}_i(1)$ and $\phi^1_i(1)$ are the respective push-forwards of $\phi^*_i(0)$ and $\phi^*_i(1)$ under the same flow of $\ti{X}_i(1)$, we have
\begin{equation}
\label{E:mF fineness control2}
\M(\partial \bar{Q}_i(1)-\partial \phi^1_i(1))\to 0, \text{ uniformly in $\alpha$ as $i\to\infty$}.
\end{equation}

Now we can apply the interpolation result \cite[Theorem 5.1]{Zhou15b} (see also \cite[Theorem 13.1]{MN14}) to $\bar{Q}_i$, which gives that for any $\eta>0$, there exist $l_\eta>0$ and $Q_i: \al(l_\eta)_0\to \C(M)$, such that
\begin{itemize}
\item[(i)] given $x\in \al(l_\eta)_0$, 
\[ \M(\partial Q_i(x)) \leq \M(\partial \bar{Q}_i(x))+\eta/2, \]
so by the same argument as in the proof of point (v) of Lemma \ref{L:interpolation1}, 
\[  \M(Q_i(x)-\bar{Q}_i(x)) \leq \eta/(2c), \]
and hence
\[ \Ac(Q_i(x))\leq \Ac(\bar{Q}_i(x))+\eta; \]

\item[(ii)] $\f(Q_i)\leq \eta$;

\item[(iii)] $\sup\{\F(\partial Q_i(x)-\partial\bar{Q}_i(x)): x\in \al(l_\eta)_0\}<\eta$.
\end{itemize}
When $\eta\to 0$, by $(i, iii)$ and \cite[2.1(20)]{P81} (see also \cite[Lemma 4.1]{MN14}), we have
\[ \lim_{\eta\to 0}\sup\{\mF(\partial Q_i(x), \partial \bar{Q}_i(x)): x\in \al(l_\eta)_0\}=0. \]

Take a sequence $\eta_i\to 0$, and denote $l_i=k_i+l_{\eta_i}+1$, then we construct $\phi_i: I(1, k_i+l_{\eta_i}+1)\to \C(M)$ by defining $\phi_i$ on each $\al(l_{\eta_i}+1)_0$ by
\[ \phi_i(x)= \left\{ \begin{array}{cl}
Q_i(3x) & \text{ for $x\in [0, 1/3]\cap \al(l_{\eta_i}+1)_0$ } \\
\phi_i^1(1) & \text{ otherwise.}
\end{array} \right. \]
The desired properties (a, b, c, d) of $\phi_i$ follow straightforwardly from (\ref{E:mF fineness control1})(\ref{E:mF fineness control2}) and the properties of $Q_j$. Since $\bar{Q}_i$ is obtained from a continuous deformation from $\phi^*_i$, a further interpolation argument shows that $S$ is homotopic to $S^*$, and hence we finish the proof of Claim 2.
\end{proof}

\section{Good replacement property and regularity}
\label{A:Good replacement property and regularity}

Here we record the notions of good replacements and the good replacement property.  Recall the following definitions by Colding-De Lellis \cite{CD03}. Consider two open subsets $W\subset\subset U\subset M^{n+1}$.

\begin{definition}
\label{D:CD good replacement}
\cite[Definition 6.1]{CD03}. Let $V \in V_n(U)$ be stationary in $U$. A stationary varifold $V' \in V_n(U)$ is said to be a {\em replacement for $V$ in $W$} if
\[ \begin{array}{cl}
& V'\lc (U\backslash \overline{W})=V\lc (U\backslash \overline{W}),\  \|V'\|(U)=\|V\|(U), \text{ and }\\
& V'\lc W \text{ is an embedded stable minimal hypersurface $\Si$ with } \partial \Si\subset \partial W.
\end{array} \]
\end{definition}

\begin{definition}
\label{D:CD good replacement property}
\cite[Definition 6.2]{CD03}. Let $V \in V_n(U)$ be stationary in $U$. $V$ is said to have the  {\em good replacement property} in $W$ if
\begin{itemize}
\item[(a)] there is a positive function $r: W\to \R$ such that for every annulus $A_{s, t}(x)\cap M\subset W$ with $0<s<t<r(x)$, there is a replacement $V'$ for $V$ in $A_{s, t}(x)\cap M$;

\item[(b)] the replacement $V'$ has a replacement $V''$ in  every annulus $A_{s, t}(y)\cap M\subset W$ with $0<s<t<r(y)$;

\item[(c)] $V''$ has a replacement $V'''$ in every annulus $A_{s, t}(z)\cap M\subset W$ with $0<s<t<r(z)$.
\end{itemize}
\end{definition}

Note that our formulations are local compared to those in \cite{CD03}. Indeed, the proofs of \cite[Proposition 6.3]{CD03} and \cite[Theorem 2.8]{DT13} are purely local, so the following proposition still holds: 

\begin{proposition}
\label{P:CD regularity}
\cite[Proposition 6.3]{CD03}, \cite[Proposition 2.8]{DT13}. When $2\leq n\leq 6$, if $V\in V_n(U)$ has the good replacement property in $W$, then $V\lc W$ is an integer multiple of some smooth embedded minimal hypersurface $\Si$.
\end{proposition}

\bibliography{min-max-cmc}

\begin{thebibliography}{10}

\bibitem{Agol-Marques-Neves16}
Ian Agol, Fernando~C. Marques, and Andr\'e Neves.
\newblock Min-max theory and the energy of links.
\newblock {\em J. Amer. Math. Soc.}, 29(2):561--578, 2016.

\bibitem{Almgren76}
F.~J. Almgren, Jr.
\newblock Existence and regularity almost everywhere of solutions to elliptic
  variational problems with constraints.
\newblock {\em Mem. Amer. Math. Soc.}, 4(165):viii+199, 1976.

\bibitem{AF62}
Frederick~Justin Almgren, Jr.
\newblock The homotopy groups of the integral cycle groups.
\newblock {\em Topology}, 1:257--299, 1962.

\bibitem{AF65}
Frederick~Justin Almgren, Jr.
\newblock The theory of varifolds.
\newblock {\em Mimeographed notes, Princeton}, 1965.

\bibitem{Arnold04}
Vladimir~I. Arnold.
\newblock {\em Arnold's problems}.
\newblock Springer-Verlag, Berlin; PHASIS, Moscow, 2004.
\newblock Translated and revised edition of the 2000 Russian original, With a
  preface by V. Philippov, A. Yakivchik and M. Peters.

\bibitem{BCE88}
J.~Lucas Barbosa, Manfredo do~Carmo, and Jost Eschenburg.
\newblock Stability of hypersurfaces of constant mean curvature in {R}iemannian
  manifolds.
\newblock {\em Math. Z.}, 197(1):123--138, 1988.

\bibitem{Bellettini-Wickramasekera17}
Costante Bellettini and Neshan Wickramasekera.
\newblock Stable cmc integral varifolds of codimension 1: regularity and
  compactness.
\newblock {\em preprint}, 2017.

\bibitem{BM82}
Pierre B\'erard and Daniel Meyer.
\newblock In\'egalit\'es isop\'erim\'etriques et applications.
\newblock {\em Ann. Sci. \'Ecole Norm. Sup. (4)}, 15(3):513--541, 1982.

\bibitem{Chambers-Liokumovich16}
Gregory~R. Chambers and Yevgeny Liokumovich.
\newblock Existence of minimal hypersurfaces in complete manifolds of finite
  volume.
\newblock {\em arXiv:1609.04058}, 2016.

\bibitem{CPP10}
Piotr~T. Chru\'sciel, Gregory~J. Galloway, and Daniel Pollack.
\newblock Mathematical general relativity: a sampler.
\newblock {\em Bull. Amer. Math. Soc. (N.S.)}, 47(4):567--638, 2010.

\bibitem{CD03}
Tobias~H. Colding and Camillo De~Lellis.
\newblock The min-max construction of minimal surfaces.
\newblock In {\em Surveys in differential geometry, {V}ol.\ {VIII} ({B}oston,
  {MA}, 2002)}, volume~8 of {\em Surv. Differ. Geom.}, pages 75--107. Int.
  Press, Somerville, MA, 2003.

\bibitem{Colding-Minicozzi}
Tobias~Holck Colding and William~P. Minicozzi, II.
\newblock {\em A course in minimal surfaces}, volume 121 of {\em Graduate
  Studies in Mathematics}.
\newblock American Mathematical Society, Providence, RI, 2011.

\bibitem{DR16}
Camillo De~Lellis and Jusuf Ramic.
\newblock Min-max theory for minimal hypersurfaces with boundary.
\newblock {\em arXiv:1611.00926}, 2016.

\bibitem{DT13}
Camillo De~Lellis and Dominik Tasnady.
\newblock The existence of embedded minimal hypersurfaces.
\newblock {\em J. Differential Geom.}, 95(3):355--388, 2013.

\bibitem{Duzaar-Steffen96}
Frank Duzaar and Klaus Steffen.
\newblock Existence of hypersurfaces with prescribed mean curvature in
  {R}iemannian manifolds.
\newblock {\em Indiana Univ. Math. J.}, 45(4):1045--1093, 1996.

\bibitem{Ginzburg96}
Viktor~L. Ginzburg.
\newblock On closed trajectories of a charge in a magnetic field. {A}n
  application of symplectic geometry.
\newblock In {\em Contact and symplectic geometry ({C}ambridge, 1994)},
  volume~8 of {\em Publ. Newton Inst.}, pages 131--148. Cambridge Univ. Press,
  Cambridge, 1996.

\bibitem{Gi}
Enrico Giusti.
\newblock {\em Minimal surfaces and functions of bounded variation}, volume~80
  of {\em Monographs in Mathematics}.
\newblock Birkh\"auser Verlag, Basel, 1984.

\bibitem{Guaraco15}
Marco A.~M. Guaraco.
\newblock Min-max for phase transitions and the existence of embedded minimal
  hypersurfaces.
\newblock {\em arXiv:1505.06698}, 2015.

\bibitem{HL75}
Reese Harvey and Blaine Lawson.
\newblock Extending minimal varieties.
\newblock {\em Invent. Math.}, 28:209--226, 1975.

\bibitem{Heinz54}
Erhard Heinz.
\newblock \"uber die {E}xistenz einer {F}l\"ache konstanter mittlerer
  {K}r\"ummung bei vorgegebener {B}erandung.
\newblock {\em Math. Ann.}, 127:258--287, 1954.

\bibitem{Hildebrandt70}
Stefan Hildebrandt.
\newblock On the {P}lateau problem for surfaces of constant mean curvature.
\newblock {\em Comm. Pure Appl. Math.}, 23:97--114, 1970.

\bibitem{HM90}
David Hoffman and William~H. Meeks, III.
\newblock The strong halfspace theorem for minimal surfaces.
\newblock {\em Invent. Math.}, 101(2):373--377, 1990.

\bibitem{HY96}
Gerhard Huisken and Shing-Tung Yau.
\newblock Definition of center of mass for isolated physical systems and unique
  foliations by stable spheres with constant mean curvature.
\newblock {\em Invent. Math.}, 124(1-3):281--311, 1996.

\bibitem{Kap90}
Nicolaos Kapouleas.
\newblock Complete constant mean curvature surfaces in {E}uclidean three-space.
\newblock {\em Ann. of Math. (2)}, 131(2):239--330, 1990.

\bibitem{Ketover16}
Daniel Ketover.
\newblock Equivariant min-max theory.
\newblock {\em arXiv:1612.08692}, 2016.

\bibitem{Ketover-Zhou15}
Daniel Ketover and Xin Zhou.
\newblock Entropy of closed surfaces and min-max theory.
\newblock {\em arXiv:1509.06238}, 2015.

\bibitem{KP92}
Steven~G. Krantz and Harold~R. Parks.
\newblock {\em A primer of real analytic functions}, volume~4 of {\em Basler
  Lehrb\"ucher [Basel Textbooks]}.
\newblock Birkh\"auser Verlag, Basel, 1992.

\bibitem{LiZ16}
Martin Li and Xin Zhou.
\newblock Min-max theory for free boundary minimal hypersurfaces i-regularity
  theory.
\newblock {\em arXiv:1611.02612}, 2016.

\bibitem{Liokumovich-Marques-Neves16}
Yevgeny. Liokumovich, Fernando~C. Marques, and Andr\'e Neves.
\newblock Weyl law for the volume spectrum.
\newblock {\em arXiv:1607.08721}, 2016.

\bibitem{Lobaton-Salamon07}
E.J. Lobaton and T.R. Salamon.
\newblock Computation of constant mean curvature surfaces: Application to the
  gas--liquid interface of a pressurized fluid on a superhydrophobic surface.
\newblock {\em Journal of Colloid and Interface Science}, 314:184--198, 2007.

\bibitem{Lopez05}
Rafael L\'opez.
\newblock Wetting phenomena and constant mean curvature surfaces with boundary.
\newblock {\em Rev. Math. Phys.}, 17(7):769--792, 2005.

\bibitem{MMP06}
F.~Mahmoudi, R.~Mazzeo, and F.~Pacard.
\newblock Constant mean curvature hypersurfaces condensing on a submanifold.
\newblock {\em Geom. Funct. Anal.}, 16(4):924--958, 2006.

\bibitem{MN14}
Fernando~C. Marques and Andr\'e Neves.
\newblock Min-max theory and the {W}illmore conjecture.
\newblock {\em Ann. of Math. (2)}, 179(2):683--782, 2014.

\bibitem{MN16}
Fernando~C. Marques and Andr\'e Neves.
\newblock Morse index and multiplicity of min-max minimal hypersurfaces.
\newblock {\em Camb. J. Math.}, 4(4):463--511, 2016.

\bibitem{MN17}
Fernando~C. Marques and Andr\'e Neves.
\newblock Existence of infinitely many minimal hypersurfaces in positive ricci
  curvature.
\newblock {\em Invent. Math.}, 2017.
\newblock doi:10.1007/s00222-017-0716-6.

\bibitem{MMPR13}
William~H Meeks~III, Pablo Mira, Joaquin Perez, and Antonio Ros.
\newblock Constant mean curvature spheres in homogeneous three-spheres.
\newblock {\em arXiv preprint arXiv:1308.2612}, 2013.

\bibitem{MMPR17}
William~H Meeks~III, Pablo Mira, Joaquin Perez, and Antonio Ros.
\newblock Constant mean curvature spheres in homogeneous three-manifolds.
\newblock {\em arXiv preprint arXiv:1706.09394}, 2017.

\bibitem{Montezuma16}
Rafael Montezuma.
\newblock Min-max minimal hypersurfaces in non-compact manifolds.
\newblock {\em J. Differential Geom.}, 103(3):475--519, 2016.

\bibitem{Mo03}
Frank Morgan.
\newblock Regularity of isoperimetric hypersurfaces in {R}iemannian manifolds.
\newblock {\em Trans. Amer. Math. Soc.}, 355(12):5041--5052, 2003.

\bibitem{Nardulli09}
Stefano Nardulli.
\newblock The isoperimetric profile of a smooth riemannian manifold for small
  volumes.
\newblock {\em Annals of global analysis and geometry, Vol. 36, No. 2 (2009),
  p.}, 36(2):111--132, 2009.

\bibitem{Novikov82}
S.~P. Novikov.
\newblock The {H}amiltonian formalism and a multivalued analogue of {M}orse
  theory.
\newblock {\em Uspekhi Mat. Nauk}, 37(5(227)):3--49, 248, 1982.

\bibitem{pacard05}
Frank Pacard.
\newblock Constant mean curvature hypersurfaces in {R}iemannian manifolds.
\newblock {\em Riv. Mat. Univ. Parma (7)}, 4*:141--162, 2005.

\bibitem{P81}
Jon~T. Pitts.
\newblock {\em Existence and regularity of minimal surfaces on {R}iemannian
  manifolds}, volume~27 of {\em Mathematical Notes}.
\newblock Princeton University Press, Princeton, N.J.; University of Tokyo
  Press, Tokyo, 1981.

\bibitem{QT07}
Jie Qing and Gang Tian.
\newblock On the uniqueness of the foliation of spheres of constant mean
  curvature in asymptotically flat 3-manifolds.
\newblock {\em J. Amer. Math. Soc.}, 20(4):1091--1110, 2007.

\bibitem{ros05}
Antonio Ros.
\newblock The isoperimetric problem.
\newblock In {\em Global theory of minimal surfaces}, volume~2 of {\em Clay
  Math. Proc.}, pages 175--209. Amer. Math. Soc., Providence, RI, 2005.

\bibitem{Rosenberg-Schneider11}
Harold Rosenberg and Matthias Schneider.
\newblock Embedded constant-curvature curves on convex surfaces.
\newblock {\em Pacific J. Math.}, 253(1):213--218, 2011.

\bibitem{Rosenberg-Smith16}
Harold Rosenberg and Graham Smith.
\newblock Degree theory of immersed hypersurfaces.
\newblock {\em arXiv:1010.1879v3}, 2016.

\bibitem{Schneider11}
Matthias Schneider.
\newblock Closed magnetic geodesics on {$S^2$}.
\newblock {\em J. Differential Geom.}, 87(2):343--388, 2011.

\bibitem{SS81}
Richard Schoen and Leon Simon.
\newblock Regularity of stable minimal hypersurfaces.
\newblock {\em Comm. Pure Appl. Math.}, 34(6):741--797, 1981.

\bibitem{SSY75}
Richard Schoen, Leon Simon, and Shing-Tung Yau.
\newblock Curvature estimates for minimal hypersurfaces.
\newblock {\em Acta Math.}, 134(3-4):275--288, 1975.

\bibitem{Si83}
Leon Simon.
\newblock {\em Lectures on geometric measure theory}, volume~3 of {\em
  Proceedings of the Centre for Mathematical Analysis, Australian National
  University}.
\newblock Australian National University, Centre for Mathematical Analysis,
  Canberra, 1983.

\bibitem{Sm82}
Francis~R. Smith.
\newblock {\em On the existence of embedded minimal 2-spheres in the 3-sphere,
  endowed with an arbitrary Riemannian metric}.
\newblock PhD thesis, Phd thesis, Supervisor: Leon Simon, University of
  Melbourne, 1982.

\bibitem{Song17}
Antoine Song.
\newblock Local min-max surfaces and strongly irreducible minimal heegaard
  splittings.
\newblock {\em arXiv:1706.01037}, 2017.

\bibitem{struwe85}
Michael Struwe.
\newblock Large {$H$}-surfaces via the mountain-pass-lemma.
\newblock {\em Math. Ann.}, 270(3):441--459, 1985.

\bibitem{struwe88}
Michael Struwe.
\newblock The existence of surfaces of constant mean curvature with free
  boundaries.
\newblock {\em Acta Math.}, 160(1-2):19--64, 1988.

\bibitem{Wente86}
Henry~C. Wente.
\newblock Counterexample to a conjecture of {H}. {H}opf.
\newblock {\em Pacific J. Math.}, 121(1):193--243, 1986.

\bibitem{White10}
Brian White.
\newblock The maximum principle for minimal varieties of arbitrary codimension.
\newblock {\em Comm. Anal. Geom.}, 18(3):421--432, 2010.

\bibitem{wick14}
Neshan Wickramasekera.
\newblock A general regularity theory for stable codimension 1 integral
  varifolds.
\newblock {\em Ann. of Math. (2)}, 179(3):843--1007, 2014.

\bibitem{Ye91}
Rugang Ye.
\newblock Foliation by constant mean curvature spheres.
\newblock {\em Pacific J. Math.}, 147(2):381--396, 1991.

\bibitem{Zhou15b}
Xin Zhou.
\newblock Min-max hypersurface in manifold of positive {R}icci curvature.
\newblock {\em J. Differential Geom.}, 105(2):291--343, 2017.

\end{thebibliography}
\bibliographystyle{plain}

\end{document}